\documentclass[12pt]{article}
\usepackage{amssymb}
\usepackage{amsmath,amsthm}
\usepackage[latin1]{inputenc}
\usepackage{graphicx}
\usepackage{enumerate}
\usepackage{tikz}
\usepackage{tkz-graph}
\usepackage{mathrsfs}
\usepackage{verbatim}

\usepackage[pdftex,
            pdfauthor={Juan Alberto Rodriguez-Velazquez and...},
            pdftitle={On the $(k,t)$-metric dimension of a graph},
            pdfsubject={Article},
            pdfkeywords={k-metric dimension; k-adjacency dimension; metric dimension; metric spaces; lexicographic product, graph theory}]{hyperref}

\hypersetup{colorlinks=true, linkcolor=blue, citecolor=blue, urlcolor=blue}

\newtheorem*{property*}{Property $\mathcal{P}$}
\newtheorem{remark}{Remark}

\newtheorem{theorem}[remark]{Theorem}
\newtheorem{proposition}[remark]{Proposition}
\newtheorem{corollary}[remark]{Corollary}

\newcommand{\Sd}{\operatorname{Sd}}

\title{The $k$-metric dimension of graphs: a general approach}

\author{ A. Estrada-Moreno$^{(1)}$, I. G. Yero$^{(2)}$ and  J. A. Rodr\'{i}guez-Vel\'{a}zquez$^{(1)}$
\\
$^{(1)}${\small Departament d'Enginyeria Inform\`atica i Matem\`atiques,}\\
{\small Universitat Rovira i Virgili,}  {\small Av. Pa\"{\i}sos
Catalans 26, 43007 Tarragona, Spain.} \\{\small
 alejandro.estrada\@@urv.cat, juanalberto.rodriguez\@@urv.cat}\\
    $^{(2)}${\small Departamento de Matem\'aticas, Escuela Polit\'ecnica Superior de Algeciras}\\
{\small Universidad de C\'adiz,} {\small
Av. Ram\'on Puyol s/n, 11202 Algeciras, Spain.} \\ {\small
ismael.gonzalez\@@uca.es}
}

\begin{document}
\maketitle

\begin{abstract}
Let $(X,d)$ be a metric space. A set $S\subseteq X$ is said to be a $k$-metric generator for $X$ if and only if  for any pair of different points $u,v\in X$, there exist at least $k$ points  $w_1,w_2, \ldots w_k\in S$ such that $d(u,w_i)\ne d(v,w_i),\; \mbox{\rm for all}\; i\in \{1, \ldots k\}.$ Let $\mathcal{R}_k(X)$ be the set of metric generators for $X$.
The
$k$-metric dimension  $\dim_k(X)$ of $(X,d)$ is defined as
$$\dim_k(X)=\inf\{|S|:\, S\in \mathcal{R}_k(X)\}.$$  Here, we discuss the $k$-metric
dimension of $(V,d_t)$, where $V$ is the set of vertices of a simple graph $G$ and the metric $d_t:V\times V\rightarrow \mathbb{N}\cup \{0\}$ is defined by $
d_t(x,y)=\min\{d(x,y),t\}
$ from   the geodesic distance $d$ in $G$ and a positive integer $t$.
The case $t\ge D(G)$, where $D(G)$ denotes the diameter of $G$, corresponds to the original theory of $k$-metric dimension and  the case $t=2$ corresponds to the theory of $k$-adjacency dimension. Furthermore, this approach allows us to extend the theory of $k$-metric dimension to the general case of non-necessarily connected graphs.
\end{abstract}



\section{Introduction}

The metric dimension of a general metric space was introduced in 1953 in \cite{Blumenthal1953}
but attracted little attention until, about twenty years later, it was applied to the
distances between vertices of a graph
\cite{Harary1976,Slater1975,Slater1988}.
 Since then it has been frequently used in graph
theory, chemistry, biology, robotics and many other disciplines.
In 2013, in \cite{Sheng2013}, the
theory of metric dimension was developed further for general metric spaces. More recently, this theory has been generalised  in \cite{Estrada-Moreno2013,Estrada-Moreno2014b,Estrada-Moreno2013corona}, again in the context
of graph theory, to the notion of a $k$-metric dimension, where $k$ is any positive integer,
and where the case $k = 1$ corresponds to the original theory.
The idea of the $k$-metric dimension both in the context of graph theory
and general metric spaces  was studied further in \cite{Beardon-RodriguezVelazquez2016}.
This paper deals with the problem of finding the $k$-metric dimension of graphs where the  metric used  is not necessarily the standard one. Given a positive integer $t$ and the geodesic distance $d$ in a graph $G=(V,E)$ we consider the metric  $d_t:V\times V\rightarrow \mathbb{R}$, defined by
$d_t(x,y)=\min\{d(x,y),t\}. $ In this context, $k$-metric generators are called $(k,t)$-metric generators and the $k$-metric dimension is  called $(k,t)$-metric dimension.
The case $t\ge D(G)$, where $D(G)$ denotes the diameter of $G$, corresponds to the original theory of $k$-metric dimension and  the case $t=2$ corresponds to the theory of $k$-adjacency dimension. Furthermore, we would point out that this approach allows us to extend the theory of $k$-metric dimension to the general case of non-necessarily connected graphs.

 The paper is structured as follows. In Section~\ref{sectPrelim}
 we introduce the main concepts and   present some basic results.  Then in Section~\ref{Sect(k,t)dimensional} we discuss a natural problem in the study of the $(k,t)$-metric dimension of a graph $G$ which consists of finding the largest integer $k$ such that
there exists a $(k,t)$-metric generator for $G$.
Section~\ref{bound-k_DimG}
is devoted to study the problem of computing or bounding the $(k,t)$-metric dimension.
 In particular, we give some basic bounds  and  discuss the extreme cases, we construct large families of graphs having a common $(k,t)$-metric generator and   show that for $t\ge 2$ the $(k,t)$-metric dimension of lexicographic product graphs does not depend on the value of $t$. We also show that the $(k,t)$-metric dimension of the corona product of a graph of
order $n$ and some nontrivial graph $H$ equals $n$ times the $(k,2)$-metric dimension of $H$. As a
consequence of this strong relationship, in Section \ref{SectionComputational complexity} we  show that  the problem of computing the $(k,t)$-metric dimension  is NP-hard  for the case in which $k$ is an odd integer.
 Finally, in Section~\ref{SectionFutureWorks} we discuss some problems which are derived from or  related to our previous results.

\section{Terminology and basic tools}\label{sectPrelim}
Let $(X,d)$ be a metric space. If $X$ is a finite set, we denote its
cardinality by $|X|$; if $X$ is an infinite set, we put $|X| =
+\infty$.  A set $A\subseteq X$ is called a metric generator for $(X,d)$ if and only if $d(x,a)=d(y,a)$
for all $a$ in $A$ implies that $x=y$.  Roughly speaking, if an object in
$x$ knows its distance from each point of $A$, then it knows exactly
where it is located in $X$. The class  of metric generators of $X$ is non-empty since $X$ is a \emph{metric generator} for  $(X,d)$.  A metric generator of a metric space $(X,d)$
is, in effect, a \emph{global coordinate system} on $X$. For example,
if $(x_1,\ldots,x_p)$ is an ordered metric generator of $X$, then the
map $\Delta:X \to \mathbb{R}^p$ given by
\begin{equation}\label{151123a}
\Delta(x)= \Big(d(x,x_1), \ldots,d(x,x_p)\Big)
\end{equation}
is injective (for this vector determines $x$), so that $\Delta$ is a
bijection from $X$ to a subset of $\mathbb{R}^p$, and $X$ inherits its
coordinates from this subset.

Now, a set $S\subseteq X$ is said to be a \emph{$k$-metric generator} for $X$ if and only if  for any pair of different points $u,v\in X$, there exist at least $k$ points  $w_1,w_2, \ldots w_k\in S$ such that $$d(u,w_i)\ne d(v,w_i),\; \mbox{\rm for all}\; i\in \{1, \ldots k\}.$$
Let $\mathcal{R}_k(X)$ be the set of $k$-metric generators for $X$.
The
$k$-\emph{metric dimension} $\dim_k(X)$ of $(X,d)$ is defined as
$$\dim_k(X)=\inf\{|S|:\, S\in \mathcal{R}_k(X)\}.$$  As $\inf\varnothing =
+\infty$, this means that $\dim_k(X) = +\infty$ if and only if no
finite subset of $X$ is a $k$-metric generator for $X$. A set $S$ is a $k$-\emph{metric basis} of $X$ if $S\in \mathcal{R}_k(X)$
and $|S|= \dim_k(X)$.

The $k$-metric dimension of metric spaces was studied in \cite{Beardon-RodriguezVelazquez2016} where,  for instance, it was shown  that if $U$ is any
non-empty open subset of any one of the three classical $n$-dimensional geometries of constant curvature, namely Euclidean space $\mathbb{R}^n$, Spherical space $\mathbb{S}^n$ and Hyperbolic space $\mathbb{H}^n$, then $\dim_k(U)=n+k$.
If we consider the discrete metric space  (equivalently, a complete graph), then $\dim_1(X)=|X|-1$ and
$\dim_2(X)=|X|$.
Furthermore, for $k\ge 3$ there are no $k$-metric generators for the discrete metric space. The reader is refereed to \cite{Estrada-Moreno2013,Estrada-Moreno2014b,Estrada-Moreno2013corona} for previous results on the $k$-metric dimension of graphs.

A basic and useful result  on the $k$-metric dimension of metric spaces is the following one.

\begin{theorem}[Monotony of $\dim_k(X)$ with respect to $k$ \cite{Beardon-RodriguezVelazquez2016}]\label{firstConsequenceMonotonyG}
Let $(X,d)$ be a metric space, and $k$ a positive integer. Then,
\begin{enumerate}[{\rm (i)}]
\item if ${\rm dim}_k(X) < +\infty$ then $\dim_k(X) < \dim_{k+1}(X)$;
\item
if ${\rm dim}_k(X) = +\infty$ then $\dim_{k+1}(X) = +\infty$.
\end{enumerate}
\noindent
In particular, $\dim_k(X) +1 \geq \dim_1(X)+ k$.
\end{theorem}

Given a positive integer $t$ and a metric space $(X,d)$, the function $d_t:X\times X\rightarrow \mathbb{R}$, defined by
\begin{equation}\label{distFunctG}
d_t(x,y)=\min\{d(x,y),t\}
\end{equation}
is a metric on $X$.
If $d(x,y)\le t$, then $d_t(x,y)=d(x,y)$, so that the $d_t$-metric topology coincides with the $d$-metric topology.
Furthermore, if $d(x,y)\ge 1$, then $d_1(x,y)=1$, so that if $d(x,y)\ge 1$ for all $x,y\in X$, then the $d_1$-metric topology coincides with the discrete-metric topology (equivalently, the topology of a complete graph).
The study of the $k$-metric dimension of $(X,d_t)$  was introduced in \cite{Beardon-RodriguezVelazquez2016} as a tool to study the $k$-metric dimension of the join of two metric spaces, and it was introduced previously in the particular context of graphs \cite{RV-F-2013}.

The next result shows that the $k$-metric dimension
of a single metric space varies when we distort the metric from $d$ to $d_t$ as above. From now on, the $k$-metric dimension of $(X,d_t)$ will be denoted by $\dim_k^t(X)$.

\begin{theorem}{\rm \cite{Beardon-RodriguezVelazquez2016}}\label{remarkMonotonyOnT}
Let $(X,d)$ be a metric space, and $k$ a positive integer, and
suppose that $0 < s < t$. Then

\begin{equation}\label{160220}
{\rm dim}^s_k(X) \geq {\rm dim}^t_k(X) \geq {\rm dim}_k(X).
\end{equation}
However, it can happen that
\begin{equation}\label{160128b}
\lim_{t\to +\infty}\ {\rm dim}^t_k(X) > {\rm dim}_k(X).
\end{equation}
\end{theorem}

If $X$ has diameter $D(X)\le t$, then the equalities in \eqref{160220} are achieved. Furthermore, as we will show in Theorem \ref{lexiSame_Metric_Adj_dim}, the equalities in  \eqref{160220} can be achieved for some metric spaces of diameter $D(X)> t\ge 2$.  Before giving an example which shows that \eqref{160128b} can
hold, we proceed to state the following result which shows that the
study of $\dim_k(X)$ should be restricted to the case of bounded metric spaces.

\begin{theorem}\label{Dim^tUnboundedSpaces}
For any unbounded connected metric space $(X,d)$ and any $t>0$, $\dim_1^t(X)=+\infty$.
\end{theorem}

\begin{proof}
We shall show that  $\dim_1^t(X)<+\infty$ implies that $(X,d)$ is bounded.
Assume that a finite set $S\subset X$ is a $1$-metric basis of $(X,d_t)$ and take $s_0,s_1\in S$ such that $d(s_0,s_1)=\displaystyle\max_{s\in S}\{d(s,s_0)\}$. Let $B=\{x\in X:\, d(s_0,x)\le t+d(s_0,s_1)\}$. If there are two different points  $x_1,x_2\in X\setminus B$, then $d_t(x_1,s)=d_t(x_2,s)=t$, for all $s\in S$, which is a contradiction. Hence, either $X\subset B$ or there exists $z\in X$ such that $X\setminus B=\{z\}$, in which case $X\subset \{x\in X:\, d(x,s_0)\le d(z,s_0)\} $. Therefore, $(X,d)$ is bounded.
\end{proof}

We have learned from \cite{Beardon-RodriguezVelazquez2016} that
$\dim_k(\mathbb{R}^n)=n+k$ and, according to Theorems \ref{firstConsequenceMonotonyG} and  \ref{Dim^tUnboundedSpaces}  $\dim_k^t(\mathbb{R}^n)=+\infty$,  which shows that \eqref{160128b} can
hold.

Let  $G=(V,E)$ be a simple and finite  graph. If $G$ is connected, then we consider the function $d:V\times V\rightarrow \mathbb{N}\cup \{0\}$, where $d(x,y)$ is the length of a shortest path between $x$ and $y$ and $\mathbb{N}$ is the set of positive integers. Obviously $(V,d)$ is a metric space, since $d$ is a metric on $V$. From now on, we will use the more intuitive notation $\dim_k^t(G)$ instead of $\dim_k^t(V)$.
In this context, in order to emphasize the role of $t$ we will use the terminology, $(k,t)$-metric generator, $(k,t)$-metric basis and $(k,t)$-metric dimension of $G=(V,E)$, instead of $k$-metric dimension, $k$-metric basis and $k$-metric dimension of $(V,d_{t})$, respectively.
Notice that when using the metric $d_t$
the concept of $k$-metric generator needs not be restricted to the case of connected graphs, as for any pair of vertices $x,y$ belonging to different connected components of $G$ we can assume that $d (x,y)=+\infty$ and so $d_t(x,y)=t$.
Hence, we can consider that the metric dimension of a  non-connected graph $G$ equals the $(k,t)$-metric dimension, where   $t $ is greater than or equal to the maximum diameter among the connected components of $G$.

We would point out the following dimension chain, which is a direct consequence of Theorems \ref{firstConsequenceMonotonyG} and \ref{remarkMonotonyOnT}. For any finite graph $G$ and any  integers $k\ge 1$ and $t\ge 2$,
\begin{equation}\label{InequalityChainktDim}
\dim_k(G)\le \dim_k^{t+1}(G)\le \dim_{k}^{t}(G)\le \dim_{k+1}^{t}(G)-1\le \dim^2_{k+1} (G)-1.
\end{equation}

 Throughout the paper, we will use the notation $K_n$, $K_{r,n-r}$, $C_n$, $N_n$ and $P_n$ for complete graphs, complete bipartite graphs, cycle graphs, empty graphs and path graphs of order $n$, respectively.

We use the notation $u \sim v$ if $u$ and $v$ are adjacent vertices and $G \cong H$ if $G$ and $H$ are isomorphic graphs. For a vertex $v$ of a graph $G$, $N_G(v)$ will denote the set of neighbours or \emph{open neighbourhood} of $v$ in $G$, \emph{i.e.} $N_G(v)=\{u \in V(G):\; u \sim v\}$. If it is clear from the context, we will use the notation $N(v)$ instead of $N_G(v)$. The \emph{closed neighbourhood} of $v$ will be denoted by $N[v]=N(v)\cup \{v\}$. Two vertices $x,y$ are called \emph{twins} if $N(x)=N(y)$ or $N[x]=N[y]$.

For the remainder of the paper, definitions will be introduced whenever a concept is needed.

\section{On $(k,t)$-metric dimensional graphs}\label{Sect(k,t)dimensional}
In this section we discuss a natural problem in the study of the $k$-metric dimension of a metric
space $(X,d_t)$ which consists of finding the largest integer $k$ such that
there exists a $k$-metric generator for $X$.
We say that a  graph $G$ is \emph{$(k,t)$-metric dimensional} if $k$ is the largest integer such that there exists a $(k,t)$-metric basis of $G$. Notice that if $G$ is a $(k,t)$-metric dimensional graph, then for each positive integer $r\le k$, there exists at least one $(r,t)$-metric basis of $G$.

Given a graph $G$ and two different vertices $x,y\in V(G)$, we denote by $\mathcal{D}_{G,t}(x,y)$ the set of vertices that distinguish the pair $x,y$ with regard to the metric $d_t$, \textit{ i.e.},  $$\mathcal{D}_{G,t}(x,y)=\{z\in V:\; d_t(z,x)\ne d_t(z,y)\}.$$
Throughout the article, if the graph $G$ is clear from the content, then we will just write $\mathcal{D}_t(x,y)$.

Note that a set  $S\subseteq V$ is a $(k,t)$-metric generator for $G=(V,E)$ if $|\mathcal{D}_t(x,y)\cap S|\ge k$ for every two different vertices $x,y\in V$. It can also be noted that two different vertices $x,y\in V$ belong to the same twin equivalence class of $G$ if and only if $\mathcal{D}_t(x,y)=\{x,y\}$. By simplicity, if $G$ has diameter $D(G)$ and it is clear from the context that $t\ge D(G)$, then we will use the notation $\mathcal{D}(x,y)$ instead of $\mathcal{D}_t(x,y)$.

Since for every pair of different vertices $x,y\in V$ we have that $|\mathcal{D}_t(x,y)|\ge 2$, it follows that the whole vertex set $V$ is a $(2,t)$-metric generator for $G$ and, as a consequence, we deduce that every graph $G$ is $(k,t)$-metric dimensional for some $k\ge 2$. On the other hand, for any graph $G=(V,E)$ of order $n\ge 3$, there exists at least one vertex $v\in V$ and two vertices $x,y\in V$ such that $\{x,y\}\in N_G(v)$ or $d_t(x,v)=d_t(y,v)=t$, so that $v\notin \mathcal{D}_t(x,y)$ and, as a result,  there is no $n$-metric dimensional graph of order $n\ge 3$. Comments above are emphasized  in the next remark.

\begin{remark}\label{remarkKMetricG}
Let $t$ be a positive integer and let $G$ be a $(k,t)$-metric dimensional graph of order $n\ge 2$. If $n\ge 3$, then $2\le k\le n-1$. Moreover, $G$ is $(n,t)$-metric dimensional if and only if $G\cong K_2$ or $G\cong N_2$.
\end{remark}

We define the following parameter $\mathfrak{d}_t(G)=\displaystyle\min_{x,y\in V}\{|\mathcal{D}_t(x,y)|\}.$ The next general result  was stated in \cite{Estrada-Moreno2014a} for the particular case of $t=2$ and also in \cite{Estrada-Moreno2013} for $t\ge D(G)$.

\begin{theorem}\label{theokdimensionalG}
Any graph $G$ of order $n\ge 2$  is $(\mathfrak{d}_t(G),t)$-metric dimensional and the time complexity of computing $\mathfrak{d}_t(G)$  is $O(n^3)$.
\end{theorem}

\begin{proof}
If $G=(V,E)$ is a $(k,t)$-metric dimensional graph, then for any $(k,t)$-metric basis $B$ and any pair of different vertices $x,y\in V$, we have
 $|B\cap {\cal D}_t(x,y)|\ge k$. Thus, $k\le \mathfrak{d}_t(G)$. Now we suppose that $k<\mathfrak{d}_t(G)$. In such a case, for every $x_i,x_j\in V$ such that $|B\cap {\cal D}_t(x_i,x_j)|=k$, the set  ${\cal D}_t(x_i,x_j)\setminus B$ must not be empty, so that  the set  $B\cup \{z\in {\cal D}_t(x_i,x_j)\setminus B:\,  |B\cap {\cal D}_t(x_i,x_j)|=k\}$
  is a $(k+1,t)$-metric generator for $G$, which is a contradiction. Therefore, $k=\mathfrak{d}_t(G).$


We now proceed to prove that  the time complexity of computing $\mathfrak{d}_t(G)$  is $O(n^3)$. We assume that the graph $G$ is represented by its adjacency matrix ${\bf A_G}$.
Hence, the problem is reduced to finding the value of $\mathfrak{d}_t(G)$. To this end, we can initially compute the distance matrix ${\bf D_G}$ from the matrix ${\bf A_G}$ by using the well-known Floyd-Warshall algorithm \cite{Roy1959, Warshall1962}, which has time complexity $O(n^3)$. The distance matrix ${\bf D_G}$ is symmetric of order $n\times n$ whose rows and columns are labelled by vertices, with entries between $0$ and $n-1$ (or $+\infty$ if $G$ is not connected). Now observe that for every $x,y\in V(G)$ we have that $z\in\mathcal{D}_t(x,y)$ if and only if $\min\{{\bf D_G}(x,z),t\}\ne \min\{{\bf D_G}(x,z),t\}$.

Given the distance matrix of $G$, computing how many vertices belong to $\mathcal{D}_t(x,y)$ for each of the $\binom{n}{2}$ pairs $x,y\in V$ can be checked in linear time. Therefore, the overall running time of such a process is bounded by the cubic time of the Floyd-Warshall algorithm.
\end{proof}

As Theorem \ref{theokdimensionalG} shows, in general, the problem of computing
$\mathfrak{d}_t(G)$ is very easy to solve. Even so,  it would be desirable to obtain some general results on this subject. In this section we restrict ourself to   discuss the extreme cases $\mathfrak{d}_t(G)=2$ and $\mathfrak{d}_t(G)=n-1$, and to study the parameter $\mathfrak{d}_t(G)$ for the particular case of paths and cycles.

If two vertices $u,v$ of $G$   belong to the same twin equivalence class, then  $\mathcal{D}_t(u,v)=\{u,v\}$, and as a consequence, we deduce the following result.

\begin{corollary}\label{coro2DimensionalG}
A graph $G$ is $(2,t)$-metric dimensional if and only if $t=1$ or there are at least two vertices of $G$ belonging to the same twin equivalence class.
\end{corollary}

An example of a $(2,t)$-metric dimensional graph is the star  $K_{1,n-1}$, whose $(2,t)$-metric dimension is $\dim_{2}^t(K_{1,n-1})=n-1$ for any $t\ge 2$, while examples of trees which are not $(2,t)$-metric dimensional are the paths $P_n$ for $n\ge 4$ and $t\ge 2$, as we will show in Proposition \ref{propKPathG}.

\begin{proposition}\label{propKPathG}
Let  $n\ge 3$ and $t$ be two integers. Then the following statements hold.
\begin{enumerate}[{\rm (i)}]
\item If $2\le t \le n-2$, then $P_n$ is $(t+1,t)$-metric dimensional.
\item  If $ t \ge n-2$, then $P_n$ is $(n-1,t)$-metric dimensional.
\end{enumerate}
\end{proposition}
\begin{proof}
Since $n\ge 3$,   Remark \ref{remarkKMetricG} leads to $\mathfrak{d}_t( P_n)\in \{2,\ldots, n-1\}$.
Let $\{u_1,u_2,\dots , u_n\}$ be the set of vertices of $P_n$ where $u_i\sim u_{i+1}$, for all $i\in \{1,\dots , n-1\}.$

We now consider two cases:
\begin{enumerate}[(i)]
\item Assume that $2\le t \le n-2$.  Since $t\ge 2$, it follows that $n\ge 4$.  Since ${\cal D}_{t}(u_1,u_2)=\{u_1,\ldots , u_{t+1}\}$, we have $\mathfrak{d}_t(P_n)\le t+1$. Let $l,r\in \{1,\ldots , n\}$ be a pair of integers  different  from the pairs $1,2$ and ${n-1}, n$ such that $l<r$. We first assume that $r-l\ge t$. If $r-l\in \{t,t+1\}$, then $|{\cal D}_{t}(u_l,u_r)|\ge r-l+1\ge t+1$.     Now, if  $r-l\ge t+2$, then $|{\cal D}_{t}(u_l,u_r)|\ge r-l\ge t+1$.
We now assume that $r-l\le  t-1$. For a given vertex $u_i$ we define the ball of center $u_i$ and radius $t-1$ as $B_i=\{u_j: d(u_i,u_j)\le t-1\}$. Notice that for any vertex $u_i$, $|B_i|\ge t$ and the equality holds if and only if $i\in \{1,n\}$.
Now, since $n\ge t+2$ and  $r-l\le  t-1$, we can claim that $|B_l|\ge t+1$ or $|B_r|\ge t+1$. Hence, if $l\ne 1$ and $r\ne n$, then $|{\cal D}_{t}(u_l,u_r)|\ge |B_l\cup B_r|-1\ge 2(t+1)-(r-l+1)-1\ge t+1$. On the other side, if $l=1$, then $r\ge 3$ and so $|{\cal D}_{t}(u_l,u_r)|\ge |B_l\cup B_r|-1=|\{u_1,u_2,\ldots , u_{r+t-1}\}|-1=r+t-2\ge  t+1$. The case $r=n$ is analogous to the previous one.
Therefore,    $\mathfrak{d}_t( P_n)=t+1$.

\item Let $ t \ge n-2$. For any pair of different vertices  there exists at most one vertex  which is not able to distinguish them. Therefore, in this case
 $\mathfrak{d}_t( P_n)=n-1$.
\end{enumerate}
\end{proof}

\begin{proposition}\label{propKClyceG}
Let  $n\ge 3$ and $t$ be two integers. Then the following statements hold.
\begin{enumerate}[{\rm (i)}]
\item If $n$ is odd and $2\le t\le\dfrac{n-1}{2}$ or $n$ is even and $2\le t\le\dfrac{n-2}{2}$, then $C_n$ is $(2t,t)$-metric dimensional.
\item If $n$ is odd and $t\ge\dfrac{n-1}{2}$, then $C_n$ is $(n-1,t)$-metric dimensional.
\item If $n$ is even and $t\ge\dfrac{n-2}{2}$, then $C_n$ is $(n-2,t)$-metric dimensional.
\end{enumerate}
\end{proposition}

\begin{proof}
Since $n\ge 3$,  Remark \ref{remarkKMetricG} leads to   $2\le \mathfrak{d}_t( C_n)\le n-1$. Let $V=\{u_0,u_2,\ldots,u_{n-1}\}$ be the vertex set of the cycle $C_n$, where $u_i\sim u_{i+1}$ and the subscripts of $u_i\in V$ are taken modulo $n$.  We now consider four cases:
\begin{enumerate}[{\rm (i)}]
\item Assume that $n$ is odd and $2\le t\le\dfrac{n-1}{2}$. Since $t\ge 2$, we have that  $n\ge 5$, and from  $\mathcal{D}_t(u_i,u_{i+1})=\{u_{i-(t-1)},\ldots,u_{i+t}\}$ we deduce that $\mathfrak{d}(C_n)\le 2t$. Let $l,r\in\{0,\ldots,n-1\}$ be two integers such that $l<r$ and $r-l\le\dfrac{n-1}{2}$. If $r-l<\dfrac{n-1}{2}$, then $\{u_{l-(t-1)},\ldots,u_l\}\cup \{u_r,\ldots,u_{r+(t-1)}\}\subseteq\mathcal{D}_t(u_l,u_r)$, and as a consequence, $|\mathcal{D}_t(u_l,u_r)|\ge 2t$. If $r-l=\dfrac{n-1}{2}$, then $\{u_{l-(t-2)},\ldots,u_l,u_{l+1}\}\cup \{u_{r-1},u_r,\ldots,u_{r+(t-2)}\}\subseteq\mathcal{D}_t(u_l,u_r)$, and thus, $|\mathcal{D}_t(u_l,u_r)|\ge 2t$ again.  Therefore, $\mathfrak{d}_t( C_n)=2t$. The case $n$ is even and $2\le t\le\dfrac{n-2}{2}$ is completely analogous to the previous one.
\item $n$ is odd and $t\ge\dfrac{n-1}{2}$. For any pair of different vertices there exists exactly one vertex which is not
able to distinguish them. Therefore, $\mathfrak{d}_t( C_n)=n-1$.
\item $n$ is even and $t\ge\dfrac{n-2}{2}$. For any pair of vertices $u_i,u_j\in V$, such that $d(u_i,u_j)=2l$, we can take a vertex $u_r$ such that $d(u_i,u_r)=d(u_j,u_r)=l$. So,  ${\cal D}_t(u_i,u_j)=V\setminus \{u_r,u_{r+\frac{n}{2}}\}$. On the other hand, if $d(u_i,u_j)$ is odd, then $|{\cal D}_t(u_i,u_j)|\ge n-2$.  Therefore, $\mathfrak{d}_t( C_n)=n-2$.
\end{enumerate}
\end{proof}

Once presented the two propositions above, we are now ready to present the characterization of $(n-1,t)$-metric dimensional graphs.

\begin{theorem}\label{n-1_Dimensional}
 A graph $G$ of order $n\ge 3$  is $(n-1,t)$-metric dimensional if and only if $G\cong P_n$ for $n\le t+2$, or $G\cong C_n$ for  an odd integer $n\le 2t+1$, or $G\cong K_1\cup K_2$, or $G\cong N_3$.
\end{theorem}

\begin{proof}
 Since $n\ge 3$,  Remark \ref{remarkKMetricG} leads to   $\mathfrak{d}_t(G)\in \{2,\ldots, n-1\}$.
If $G$ is a path of order $n\le t+2$, then by Proposition \ref{propKPathG} we have that $G$ is $(n-1,t)$-metric dimensional. If $G$ is a cycle of odd order $n\le 2t+1$, then by Proposition \ref{propKClyceG} it follows that $G$ is $(n-1,t)$-metric dimensional. If $G\cong K_1\cup K_2$ or $G\cong N_3$, then it is straightforward to see that $G$ is $(n-1,t)$-metric dimensional.

On the other side, let $G$ be a graph such that $\mathfrak{d}_t(G)=n-1$.
 Hence, for every pair of different vertices $x,y\in V(G)$ there exists at most one vertex which does not distinguish $x,y$. Suppose $G$ has maximum degree $\Delta(G)>2$ and let $v\in V(G)$ such that $\{u_1,u_2,u_3\}\subseteq N(v)$. Figure \ref{figCases} shows all the possibilities for the links between these four vertices. Figures \ref{figCases} (a), \ref{figCases} (b) and \ref{figCases} (d) show that $v,u_1$ do not distinguish $u_2,u_3$. Figure \ref{figCases} (c) shows that $u_1,u_2$ do not distinguish $v,u_3$. This analysis shows that  $\mathfrak{d}_t(G)\le n-2$, which is a contradiction and, as a consequence,  $\Delta(G)\le 2$. If $G$ is connected, then we have that $G$ is either a path or a cycle, and by Propositions \ref{propKPathG} and \ref{propKClyceG}, we deduce that
 $G\cong P_n$ for $n\le t+2$, or $G\cong C_n$ for an odd integer  $n\le 2t+1$.
 From now on we assume that $G$ is not connected.
Notice that  each connected component is either a path, or a cycle or an isolated vertex. If one of the connected components $G'$ has order at least three, then there exist three vertices  $v,x,y$ such that $N(v)=\{x,y\}$. Neither $v$ nor the vertices of connected components different from $G'$ are able to distinguish $x$ and $y$, which is a contradiction.  Thus, each connected component of $G$ has maximum degree at most one. Therefore, if $K_2$ is a connected component of $G$, then $G\cong K_1\cup K_2$, and if $G$ is empty, then $n=3$.
\end{proof}

\begin{figure}[h]
\centering
\begin{tikzpicture}[transform shape, inner sep = .7mm]
\node [draw=black, shape=circle, fill=white,text=black] (ua1) at (180:1.42) {};
\node [scale=.8] at ([yshift=-.3 cm]ua1) {$u_1$};
\node [draw=black, shape=circle, fill=white,text=black] (va) at (0,0) {};
\node [scale=.8] at ([yshift=-.3 cm]va) {$v$};
\node [draw=black, shape=circle, fill=white,text=black] (ua2) at (315:1.42) {};
\node [scale=.8] at ([yshift=-.3 cm]ua2) {$u_2$};
\node [draw=black, shape=circle, fill=white,text=black] (ua3) at (45:1.42) {};
\node [scale=.8] at ([yshift=.3 cm]ua3) {$u_3$};
\draw[black] (ua1) -- (va)  -- (ua2);
\draw[black] (va)  -- (ua3);
\node at ([yshift=-1.8 cm]va) {(a)};

\node [draw=black, shape=circle, fill=white,text=black, xshift=3.5 cm] (ub1) at (180:1.42) {};
\node [scale=.8] at ([yshift=-.3 cm]ub1) {$u_1$};
\node [draw=black, shape=circle, fill=white,text=black, xshift=3.5 cm] (vb) at (0,0) {};
\node [scale=.8] at ([yshift=-.3 cm]vb) {$v$};
\node [draw=black, shape=circle, fill=white,text=black, xshift=3.5 cm] (ub2) at (315:1.42) {};
\node [scale=.8] at ([yshift=-.3 cm]ub2) {$u_2$};
\node [draw=black, shape=circle, fill=white,text=black, xshift=3.5 cm] (ub3) at (45:1.42) {};
\node [scale=.8] at ([yshift=.3 cm]ub3) {$u_3$};
\draw[black] (ub1) -- (vb)  -- (ub2) -- (ub3) -- (vb);
\node at ([yshift=-1.8 cm]vb) {(b)};

\node [draw=black, shape=circle, fill=white,text=black] (vc) at ([xshift=2.6 cm]ub2) {};
\node [scale=.8] at ([yshift=-.3 cm]vc) {$v$};
\node [draw=black, shape=circle, fill=white,text=black] (uc1) at ([shift=({180:1.42})]vc) {};
\node [scale=.8] at ([yshift=-.3 cm]uc1) {$u_1$};
\node [draw=black, shape=circle, fill=white,text=black] (uc2) at ([shift=({0:1.42})]vc) {};
\node [scale=.8] at ([yshift=-.3 cm]uc2) {$u_2$};
\node [draw=black, shape=circle, fill=white,text=black] (uc3) at ([shift=({90:1.42})]vc) {};
\node [scale=.8] at ([yshift=.3 cm]uc3) {$u_3$};
\draw[black] (uc1) -- (vc)  -- (uc2) -- (uc3) -- (uc1);
\draw[black] (vc) -- (uc3);
\node at ([yshift=-.8 cm]vc) {(c)};

\node [draw=black, shape=circle, fill=white,text=black] (vd) at ([xshift=7.3 cm]vb) {};
\node [scale=.8] at ([xshift=.3 cm]vd) {$v$};
\node [draw=black, shape=circle, fill=white,text=black] (ud1) at ([shift=({225:1.42})]vd) {};
\node [scale=.8] at ([yshift=-.3 cm]ud1) {$u_1$};
\node [draw=black, shape=circle, fill=white,text=black] (ud2) at ([shift=({315:1.42})]vd) {};
\node [scale=.8] at ([yshift=-.3 cm]ud2) {$u_2$};
\node [draw=black, shape=circle, fill=white,text=black] (ud3) at ([shift=({90:1.42})]vd) {};
\node [scale=.8] at ([yshift=.3 cm]ud3) {$u_3$};
\draw[black] (ud1) -- (vd)  -- (ud2) -- (ud3) -- (ud1) -- (ud2);
\draw[black] (vd) -- (ud3);
\node at ([yshift=-1.8 cm]vd) {(d)};
\end{tikzpicture}
\caption{Possible cases for a vertex $v$ with three neighbours  $u_1,u_2,u_3$.}\label{figCases}
\end{figure}
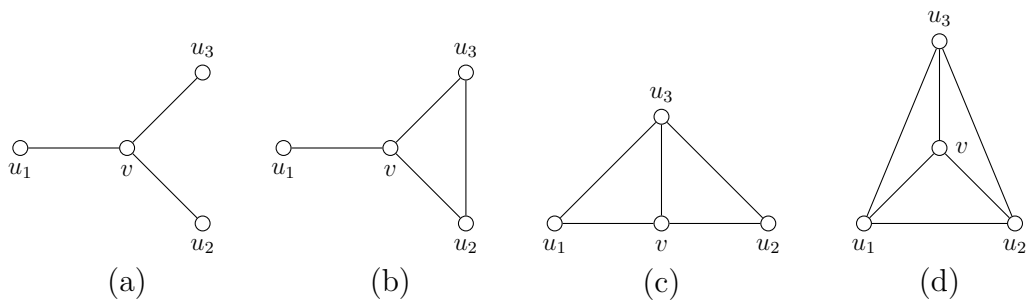

\section{On the $(k,t)$-metric dimension of graphs}\label{bound-k_DimG}

In this section we study the problem of computing or bounding the $(k,t)$-metric dimension. Since for any graph $G$ of order $n\ge 2$, we have that $\mathfrak{d}_1(G)=2$,  $\dim_1^1(G)=n-1$ and  $\dim_2^1(G)=n$, from now on we assume that $t\ge 2$. To begin with, we consider the limit case of the trivial bound $\dim_{k}^t(G)\ge k$.

\begin{theorem}\label{propValueClassic2G}
Let  $G$ be a graph of order $n\ge 2$. Then  $\dim_{k}^t(G)=k$ if and only if $k\in \{1,2\}$, $n\le t+1$ and \emph{(}$G\cong K_1\cup P_{n-1}$ or $G\cong P_{n}$\emph{)}.
\end{theorem}

\begin{proof}
It is readily seen that if $k\in \{1,2\}$, $n\le t+1$ and ($G\cong K_1\cup P_{n-1}$ or $G\cong P_{n}$), then $\dim_{k}^t(G)=k$.

Conversely, suppose that $\dim_{k}^t(G)=k$ and let $S$ be a $(k,t)$-metric basis of $G$.
Given $s\in S$ and a non-negative integer $r$, we define the set
$$\Gamma_r(s)=\{v\in V(G):\, d_t(v,s)=r\}.$$
Since $|S|=k$ and for any $x,y\in V(G)$, $|S\cap \mathcal{D}_t(x,y)|\ge k$,  we have that $S\subseteq \mathcal{D}_t(x,y)$, \textit{i.e.},  for any $s\in S$ and $x,y\in V(G)$, $d_t(s,x)\ne d_t(s,y)$. Hence, for any $s\in S$ and any non-negative integer $r$, we have $|\Gamma_r(s)|\le 1$, which implies that $n\le t+1$ and also $G\cong K_1\cup P_{n-1}$ or $G\cong P_{n}$. Notice that the vertices in $S$ must have degree at most one and so we deduce that $k=|S|\le 2$.
\end{proof}

The following result is a direct consequence of Theorems  \ref{firstConsequenceMonotonyG}  and \ref{theokdimensionalG}.

\begin{theorem}\label{TrivialUpperBound}
For any graph $G$ of order $n$ and any $k\in \{1,\dots , \mathfrak{d}_t(G)\}$,
$$\dim_k^t(G)\le n-\mathfrak{d}_t(G)+k.$$
\end{theorem}

As the following result shows, the bound above is tight.

 \begin{remark}
Let     $k\ge 1$,  $t\ge 2$ and $n\ge 3$ be three integers. Then the following statements hold.
\begin{enumerate}[{\rm (i)}]
\item {\rm \cite{Estrada-Moreno2014a, JanOmo2012}} For any $n\ge 4$, $\dim_1^2(P_n)=\left\lfloor{\frac{2n+2}{5}}\right\rfloor$, $\dim_2^2(P_n)=\left\lceil\frac{n+1}{2}\right\rceil$ and $\dim_3^2(P_n)=n-\left\lfloor\frac{n-4}{5}\right\rfloor$.
\item If $ t \le n-2$ and $k\le t+1$, then $k+1\le \dim_k^t(P_n)\le n-t+k-1$.
\item Let $k+1\le n\le 2t-k+3$. If $k\ge 3$ or $n\ge t+2$, then $\dim_{k}^t(P_n)=k+1.$
\item {\rm \cite{Estrada-Moreno2014a, JanOmo2012}} For any $n\ge 4$, $\dim_1^2(C_n)=\left\lfloor{\frac{2n+2}{5}}\right\rfloor$, $\dim_2^2(C_n)=\left\lceil\frac{n}{2}\right\rceil$, $\dim_3^2(C_n)=n-\left\lfloor\frac{n}{5}\right\rfloor$ and $\dim_4^2(C_n)=n$.
\item Let $ k\le 2t$. If $n$ is odd and $t\le\frac{n-1}{2}$ or $n$ is even and $ t\le\frac{n-2}{2}$, then   $k+1\le \dim_k^t(C_n)\le n-2t+k$.
\item If $n$ is odd,  $t\ge\frac{n-1}{2}$ and $ k\le n-1$, then $\dim_k^t(C_n)=  k+1$.
\item  Let $n$ even and $t\ge\frac{n-2}{2}$. If $ k\le \frac{n-2}{2}$, then  $\dim_k^t(C_n)=k+1$ and, if $\frac{n}{2}\le k  \le n-2$, then  $\dim_k^t(C_n)=k+2$.
\end{enumerate}
\end{remark}

\begin{proof}
By combining  Proposition \ref{propKPathG}    and Theorems \ref{propValueClassic2G}  and \ref{TrivialUpperBound}  we deduce (ii) and by combining  Proposition \ref{propKClyceG}  and Theorems \ref{propValueClassic2G}  and \ref{TrivialUpperBound}  we deduce (v) and (vi).

We now proceed to prove (iii). Let $V=\{v_1,v_2,\dots, v_n\}$ be the vertex set of $P_n$, where $v_i\sim v_{i+1}$, for all $i\in \{1,\dots,n-1\}$,    and set $$\displaystyle S=\left\{v_{\left\lceil\frac{n}{2}\right\rceil-\left\lfloor\frac{k}{2}\right\rfloor},v_{\left\lceil\frac{n}{2}\right\rceil-\left\lfloor\frac{k}{2}\right\rfloor+1},\ldots,v_{\left\lceil\frac{n}{2}\right\rceil+\left\lceil\frac{k}{2}\right\rceil}\right\}.$$ Note that $|S|=k+1$. If $n\le 2t-k+3$, then for any pair of different vertices $u,v\in V $ there exists at most one vertex $w\in S$ such that $d_t(w,u)=d_t(w,v)$. Thus, for every pair of different vertices $x,y\in V $, there exist at least $k$ vertices of $S$ such that they distinguish $x,y$. So $S$ is a $(k,t)$-metric generator for $P_n$. Therefore, $\dim_{k}(P_n,t)\le |S|=k+1$ and, consequently, (iii) follows by Theorem   \ref{propValueClassic2G}.

Finally, we   proceed to prove (vii).
By combining  Proposition \ref{propKClyceG}  and Theorems \ref{propValueClassic2G}  and \ref{TrivialUpperBound}  we deduce that for
$n$  even,   $t\ge\frac{n-2}{2}$ and $1\le k\le n-2$, we have $k+1\le \dim_k^t(C_n)\le k+2$.

Let $S$ be $(k,t)$-metric basis of $C_n$. Notice that $|S|=k+1$ or $|S|=k+2$. If $k\ge\frac{n}{2}$, then there are two antipodal vertices, $u$ and $v$, belonging to $S$.    Thus, there exist at least two vertices of $C_n$ which are not  distinguished neither by $u$ nor by $v$, which implies  that $|S|=k+2$.

Suppose that $k<\frac{n}{2}$.  Since $t\ge\frac{n-2}{2}$, any set of $k+1$ consecutive vertices of $C_n$ is a $(k,t)$-metric generator and, in such a case $\dim_k^t(C_n)=k+1$. Therefore, the proof of (vii) is complete.
\end{proof}

Let $\mathfrak{D}_{t,k}(G)$ be the set obtained as the union of the sets ${\cal D}_t(x,y)$ that distinguish a pair of different vertices $x,y$ whenever $|{\cal D}_t(x,y)|=k$, {\it i.e.},
$$\mathfrak{D}_{t,k}(G)=\bigcup_{|{\cal D}_t(x,y)|=k}{\cal D}_t(x,y).$$

By a reasoning similar to that described in the proof of Theorem \ref{theokdimensionalG}
 we can check that  the time complexity of computing $\mathfrak{D}_{t,k}(G)$  is $O(n^3)$.

\begin{remark}\label{remTaukG}
For any $(\mathfrak{d}_t(G),t)$-metric basis $B$ of a graph $G$ we have $\mathfrak{D}_{t,\mathfrak{d}_t(G)}(G)\subseteq B$, and as a consequence, $\dim_{\mathfrak{d}_t(G)}^t(G)\geq |\mathfrak{D}_{t,\mathfrak{d}_t(G)}(G)|$
\end{remark}

\begin{proof}
Since every pair of different vertices $x,y$ is distinguished only by the elements of ${\cal D}_t(x,y)$, if $|\mathcal{D}_t(u,v)|=\mathfrak{d}_t(G)$ for some $u,v$ of $G$, then for any $(\mathfrak{d}_t(G),t)$-metric basis $B$ we have $\mathcal{D}_t(u,v)\subseteq B$, and as a consequence,
$\mathfrak{D}_{t,\mathfrak{d}_t(G)}(G)\subseteq B$. Therefore, the result follows.
\end{proof}

The bound given in  Remark \ref{remTaukG} is tight. For instance, for $t\ge D(G)$ we already shown in \cite{Estrada-Moreno2013} that there exists a family of trees attaining this bound for every $k$. Other examples for any positive integer $t\ge 2$ can be derived from the following result.

\begin{theorem}\label{theoDimknG}
Let  $G=(V,E)$ be a graph   of order $n\ge 2$. Then the following assertions  hold.
\begin{enumerate}[{\rm (i)}]
\item  $\dim_{\mathfrak{d}_t(G)}^t(G)=n$ if and only if $\mathfrak{D}_{t,\mathfrak{d}_t(G)}(G)=V$.

\item  If $|\mathfrak{D}_{t,\mathfrak{d}_t(G)}(G)|=n-1$, then $\dim_{\mathfrak{d}_t(G)}^t(G)=n-1$.
\end{enumerate}
\end{theorem}

\begin{proof}
Suppose that $\mathfrak{D}_{t,\mathfrak{d}_t(G)}(G)=V$. Since  $\dim_{\mathfrak{d}_t(G)}^t(G)\le n$, by Remark \ref{remTaukG} we obtain that $\dim_{\mathfrak{d}_t(G)}^t(G)=n$.

On the other hand, assume that  $\dim_{\mathfrak{d}_t(G)}^t(G)=n$. Note that for every
$a,b\in V $, we have $|{\cal D}_t(a,b)|\ge \mathfrak{d}_t(G)$. If there exists at least one vertex $x\in V$ such that $x\notin \mathfrak{D}_{t,\mathfrak{d}_t(G)}(G)$, then for every $a,b\in V $, we have $|{\cal D}_t(a,b)\setminus \{x\}|\ge \mathfrak{d}_t(G)$ and, as a consequence, $V \setminus \{x\}$ is a  $(\mathfrak{d}_t(G),t)$-metric generator for $G$, which is a contradiction. Therefore, $ \mathfrak{D}_{t,\mathfrak{d}_t(G)}(G)=V$.

Finally,  if $|\mathfrak{D}_{t,\mathfrak{d}_t(G)}(G)|=n-1$, by Remark \ref{remTaukG} and (i) we conclude that (ii) follows.
\end{proof}

\begin{corollary}\label{remarkDim2nG}
Let $G$ be a graph of order $n\geq 2$. Then $\dim_2^t(G)=n$ if and only if every vertex of $G$  belongs to a non-singleton twin equivalence class.
\end{corollary}

We will show other examples of graphs that satisfy Theorem \ref{theoDimknG} for $k\ge 3$. Let $W_{1,n}=K_1+C_n$ be the wheel graph and $F_{1,n}=K_1+P_n$ be the fan graph. Since $V(F_{1,4})=\mathfrak{D}_{3,t}(F_{1,4})$ and $V(W_{1,5})=\mathfrak{D}_{4,t}(W_{1,5})$, by Theorem \ref{theoDimknG} we have that $\dim_3^t(F_{1,4})=5$ and $\dim_4^t(W_{1,5})=6$.

\subsection{Large families of graphs having a common $(k,t)$-metric generator} \label{SectionFamilies(k,t)Metric}

The aim of this subsection is to show examples   of large families of graphs (defined on a common vertex set) having a common $(k,t)$-metric generator.
We will use the notation  $d_{G,t}(x,y)$ instead of $d_t(x,y)$ with the aim of emphasising that the distance has been defined on $G$.

Let $B$ be a $(k,t)$-metric basis of a graph $G=(V,E)$ of diameter $D(G)$, and let $D_t(G)=\min\{D(G),t\}$. For any $r \in \{0,1,\ldots,D_t(G)\}$ we define the set $${\rm \bf B}_{r}(B)=\displaystyle\bigcup_{x\in B}\{y\in V:\; d_{G,t}(x,y)\le r\}.$$\label{ball}
In particular, ${\rm {\bf  B}}_{0}(B)=B$ and ${\rm {\bf  B}}_{1}(B)=\displaystyle\bigcup_{x\in B}N_G[x] $. Moreover,  since $B$ is a $(k,t)$-metric basis of $G$, $|{\rm {\bf  B}}_{D_t(G)-1}(B)|\ge |V|-1$.

Assume that  $G\not\cong K_n$. Given a $(k,t)$-metric basis $B$ of $G$ we say that a graph $G'=(V,E')$ belongs to the  family  ${\mathcal{G}}_B(G)$\label{familCommonGen} if and only if $N_{G'}(v)=N_G(v)$, for every $v\in {\rm {\bf  B}}_{D_t(G)-2}(B)$. In particular, if $t=2$, then  $G'=(V,E')$ belongs to the  family  ${\mathcal{G}}_B(G)$ if and only if $N_{G'}(x)=N_G(x)$, for every $x\in B$. Moreover, if $G$ is a complete graph, we define ${\mathcal{G}}_B(G)=\{G\}$. By the definition of ${\mathcal{G}}_B(G)$, we deduce the following remark.

\begin{remark}\label{remarkSameDistanceIsomorphic}
Let $B$ be a $(k,t)$-metric basis of a connected  graph $G$,  and let $G'\in {\mathcal{G}}_B(G)$. Then for any $b\in B$ and $v\in {\rm {\bf  B}}_{D_t(G)-1}(B)$,
$d_{G,t}(b,v)=d_{G',t}(b,v).$ 
\end{remark}

Notice that if  ${\rm {\bf{B}}}_{D_t(G)-2}(B)\subsetneq V$, then  any graph $G'\in {\mathcal{G}}_B(G)$ is isomorphic to a graph $G^*=(V,E^*)$ whose edge set $E^*$ can be partitioned into two sets $E^*_1$, $E^*_2$, where $E^*_1$ consists of all edges of $G$ having at least one vertex in ${\rm {\bf  B}}_{D(G)-2}(B)$ and $E^*_2$ is a subset of edges of a complete graph whose vertex set is $V\setminus {\rm {\bf{B}}}_{D_t(G)-2}(B)$. Hence, if $\displaystyle l={|V\setminus {\rm {\bf{B}}}_{D_t(G)-2}(B)|\choose 2}$, then ${\cal{G}}_B(G)$ contains
$2^l$ different graphs, where some of them could be isomorphic.

\begin{figure}[!ht]
\centering
\begin{tikzpicture}[transform shape, inner sep = .7mm]
\foreach \y in {0,1}
{
\foreach \x in {0,1,2}
{
\ifthenelse{\x=0 \AND \y=0}
{
\node [draw=black, shape=circle, fill=white] (v001) at (0,0) {};
\node [scale=.8] at ([yshift=-.3 cm]v001) {$v_1$};
\foreach \ind in {2,...,5}
{
\pgfmathparse{(1.5-\ind+2)*.8};
\node [draw=black, shape=circle, fill=black] (v00\ind) at ([shift=({(1.8,\pgfmathresult)})]v001) {};
\node [scale=.8] at ([yshift=-.32 cm]v00\ind) {$v_\ind$};
\draw[black] (v001) -- (v00\ind);
}
\foreach \ind in {6,...,9}
{
\pgfmathparse{(1.5-\ind+6)*1.2};
\node [draw=black, shape=circle, fill=white] (v00\ind) at ([shift=({(3.6,\pgfmathresult)})]v001) {};
\node [scale=.8] at ([yshift=-.3 cm]v00\ind) {$v_\ind$};
}
\foreach \ind in {2,3}
\draw[black] (v006) -- (v00\ind);
\foreach \ind in {3,4}
\draw[black] (v007) -- (v00\ind);
\foreach \ind in {4,5}
\draw[black] (v008) -- (v00\ind);
\foreach \ind in {2,5}
\draw[black] (v009) -- (v00\ind);
}
{
\pgfmathparse{\y*(-4.5)};
\node [draw=black, shape=circle, fill=white] (v\y\x1) at (4.5*\x,\pgfmathresult) {};
\node [scale=.8] at ([yshift=-.3 cm]v\y\x1) {$v_1$};
\foreach \ind in {2,...,5}
{
\pgfmathparse{(1.5-\ind+2)*.8};
\node [draw=black, shape=circle, fill=black] (v\y\x\ind) at ([shift=({(1.8,\pgfmathresult)})]v\y\x1) {};
\node [scale=.8] at ([yshift=-.3 cm]v\y\x\ind) {$v_\ind$};
\draw[black] (v\y\x1) -- (v\y\x\ind);
}
\foreach \ind in {6,...,9}
{
\pgfmathparse{(1.5-\ind+6)*1.2};
\node [draw=black, shape=circle, fill=white] (v\y\x\ind) at ([shift=({(3.6,\pgfmathresult)})]v\y\x1) {};
\node [scale=.8] at ([yshift=-.3 cm]v\y\x\ind) {$v_\ind$};
}
\foreach \ind in {2,3}
\draw[black] (v\y\x6) -- (v\y\x\ind);
\foreach \ind in {3,4}
\draw[black] (v\y\x7) -- (v\y\x\ind);
\foreach \ind in {4,5}
\draw[black] (v\y\x8) -- (v\y\x\ind);
\foreach \ind in {2,5}
\draw[black] (v\y\x9) -- (v\y\x\ind);
};
}
}
\node at ([yshift=-1 cm]v005) {$G$};
\draw [black] (v006.east) to[bend left] (v007.east);
\foreach \y in {0,1}
{
\foreach \x in {0,1,2}
{
\ifthenelse{\x>0 \OR \y>0}
{
\pgfmathparse{int(\y*3+\x)};
\node at ([yshift=-1 cm]v\y\x5) {$G_\pgfmathresult$};
}
{};
}
}
\draw [dashed] (v111.north) to[bend left] (v116.north);
\draw [dashed] (v121.north) to[bend left] (v126.north);
\draw [dashed] (v121.south) to[bend right] (v129.south);
\draw (v016.east) to[bend left] (v017.east);
\draw [dashed] (v017.east) to[bend left] (v018.east);
\draw (v026.east) to[bend left] (v027.east);
\draw [dashed] (v027.east) to[bend left] (v028.east);
\draw [dashed] (v028.east) to[bend left] (v029.east);
\end{tikzpicture}
\caption{$B=\{v_2,v_3,v_4,v_5\}$ is a $(2,2)$-metric basis of $G$ and
$\{G,G_1,G_2,G_4,G_5\}\subset {\cal  G}_{B}(G)$.}\label{figExamplekPermutation}
\end{figure}
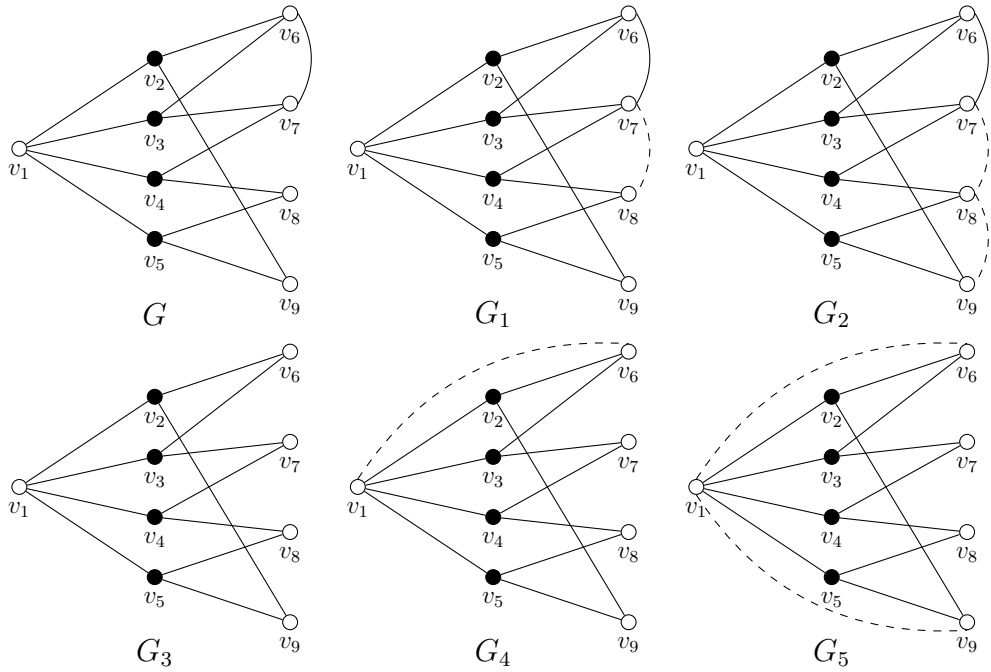

\begin{theorem}\label{Common_k_t_metric_generator}
Any $(k,t)$-metric basis $B$ of a graph $G$ is a $(k,t)$-metric generator for any graph $G'\in {\mathcal{G}}_B(G)$, and as a consequence, $$\dim_{k}^t(G')\le\dim_{k}^t(G).$$
\end{theorem}

\begin{proof}
Assume that $B$ is a $(k,t)$-metric basis of a graph $G=(V,E)$, and $G'\in {\mathcal{G}}_B(G)$. We shall show that $B$ is a $(k,t)$-metric generator for $G'$. To this end, we take two different vertices $u,v\in V$. Since $B$ is a $(k,t)$-metric basis of $G$, there exists $B_{uv}\subseteq B$ such that $|B_{uv}|\ge k$ and for every $x\in B_{uv}$ we have that $d_{G,t}(x,u)\ne d_{G,t}(x,v)$. Now, consider the following two cases for $u,v$.
\\
\\
\noindent (1) $u,v\in {\rm {\bf  B}}_{D_t(G)-1}(B)$. In this case, since for every $x\in B_{uv}$ we have that $d_{G,t}(x,u)\ne d_{G,t}(x,v)$, Remark \ref{remarkSameDistanceIsomorphic} leads to $d_{G',t}(x,u)\ne d_{G',t}(x,v)$ for every $x\in B_{uv}$.
\\
\\
\noindent (2) $u\in {\rm {\bf  B}}_{D_t(G)-1}(B)$ and $v\not\in {\rm {\bf  B}}_{D_t(G)-1}(B)$. By definition of ${\rm {\bf  B}}_{D_t(G)-1}(B)$, we deduce that $d_{G',t}(x,u)\le D_t(G)-1$ for every $x\in B_{uv}$. Since $v\not\in {\rm {\bf  B}}_{D_t(G)-1}(B)$, we have that $d_{G',t}(x,v)= D_t(G)$ for every $x\in B_{uv}$. So, $d_{G',t}(x,u)\le D_t(G)-1<D_t(G)=d_{G',t}(x,v)$ for every $x\in B_{uv}$.
\\
\\
Notice that since $B$ is a $(k,t)$-metric basis of $G$, the case $u,v\not\in {\rm {\bf  B}}_{D_t(G)-1}(B)$ is not possible. According to the two cases above, $B$ is a $(k,t)$-metric generator for $G'$. Therefore, $\dim_{k}^t(G')\le |B| =\dim_{k}^t(G)$.
\end{proof}

By Theorems \ref{propValueClassic2G} and  \ref{Common_k_t_metric_generator} we deduce the following result.

\begin{remark}\label{Corollary-Famili-dim_k_t=k+1}
Let $B$ be a $(k,t)$-metric basis  of  a graph $G$ of order $n\ge t+2$ and let $G'\in {\cal  G}_B(G)$. If $\dim_{k}^t(G)=k+1$, then $\dim_{k}^t(G')=k+1.$
\end{remark}

Figure \ref{figExamplekPermutation} shows some graphs belonging to the  family  ${\cal  G}_{B}(G)$  having a common $(2,2)$-metric generator $B=\{v_2,v_3,v_4,v_5\}$.  In fact $B$ is also a common $(2,2)$-metric basis for all graphs belonging to ${\cal  G}_{B}(G)$. In
this case, the family ${\cal  G}_{B}(G)$ contains $2^{10}=1024$ different graphs, where some of them could be isomorphic.
\subsection{The case of lexicographic product graphs}\label{(k,t)-metricDim lexicographic}

Let $G$ be a graph of order $n$, and let $\mathcal{H}=\{H_1,H_2,\ldots,H_n\}$ be an ordered family composed by $n$ graphs. The \emph{lexicographic product} of $G$ and $\mathcal{H}$ is the graph $G \circ \mathcal{H}$, such that $V(G \circ \mathcal{H})=\bigcup_{u_i \in V(G)} (\{u_i\} \times V(H_i))$ and $(u_i,v_r)(u_j,v_s) \in E(G \circ \mathcal{H})$ if and only if $u_iu_j \in E(G)$ or $i=j$ and $v_rv_s \in E(H_i)$.
Figure \ref{figExampleOfLexiFamily} shows the lexicographic product of $P_3$ and the family composed by $\{P_4,K_2,P_3\}$, and the lexicographic product of $P_4$ and the family $\{H_1,H_2,H_3,H_4\}$, where $H_1 \cong H_4 \cong K_1$ and $H_2 \cong H_3 \cong K_2$. In general, we can construct the graph $G\circ\mathcal{H}$ by taking one copy of each $H_i\in\mathcal{H}$ and joining by an edge every vertex of $H_i$ with every vertex of $H_j$ for every $u_i u_j\in E(G)$.

\begin{figure}[!ht]
\centering
\begin{tikzpicture}[transform shape, inner sep = .7mm]
\pgfmathsetmacro{\espacio}{1};
\node [draw=black, shape=circle, fill=white] (v5) at (3,0) {};
\node [draw=black, shape=circle, fill=white] (v6) at (3,3*\espacio) {};
\draw[black] (v5) -- (v6);
\foreach \ind in {1,...,4}
{
\pgfmathsetmacro{\yc}{(\ind-1)*\espacio};
\node [draw=black, shape=circle, fill=white] (v\ind) at (0,\yc) {};
\draw[black] (v5) -- (v\ind);
\draw[black] (v6) -- (v\ind);
\ifthenelse{\ind>1}
{
\pgfmathtruncatemacro{\bind}{\ind-1};
\draw[black] (v\ind) -- (v\bind);
}
{};
}
\node [draw=black, shape=circle, fill=white] (v7) at (6,0) {};
\node [draw=black, shape=circle, fill=white] (v8) at (6,1.5*\espacio) {};
\node [draw=black, shape=circle, fill=white] (v9) at (6,3*\espacio) {};
\draw[black] (v7) -- (v8) -- (v9);
\foreach \ind in {7,...,9}
{
\draw[black] (v5) -- (v\ind);
\draw[black] (v6) -- (v\ind);
}
\end{tikzpicture}
\hspace{0.5cm}
\begin{tikzpicture}[transform shape, inner sep = .7mm]

\draw(-1,-1) -- (1,1)--(1,-1)--(-1,1)--(-2,0)--(-1,-1)--(1,-1)--(2,0)--(1,1)--(-1,1)--(-1,-1);
\node  at  (0,-1.5) {};

\node [draw=black, shape=circle, fill=white] at  (-1,-1) {};
\node [draw=black, shape=circle, fill=white] at  (-1,1) {};
\node [draw=black, shape=circle, fill=white] at  (1,-1) {};
\node [draw=black, shape=circle, fill=white] at  (1,1) {};
\node [draw=black, shape=circle, fill=white] at  (-2,0) {};
\node [draw=black, shape=circle, fill=white] at  (2,0) {};

\end{tikzpicture}
\caption{The lexicographic product graphs $P_3\circ\{P_4,K_2,P_3\}$ and $P_4\circ\{H_1,H_2,H_3,H_4\}$, where $H_1 \cong H_4 \cong K_1$ and $H_2 \cong H_3 \cong K_2$.}\label{figExampleOfLexiFamily}
\end{figure}
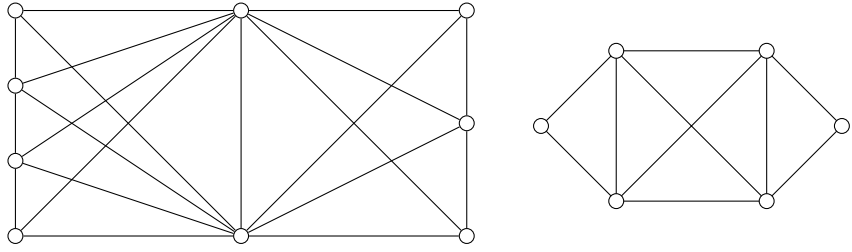

The standard concept of lexicographic product graph is the particular case when $H_i\cong H$ for every $i\in\{1,\ldots, n\}$ and it is denoted as $G\circ H$ for simplicity.  Another particular case of lexicographic product graphs is the join graph. The \emph{join graph} $G+H$\label{g join} is defined as the graph obtained from disjoint graphs $G$ and $H$ by taking one copy of $G$ and one copy of $H$ and joining by an edge each vertex of $G$ with each vertex of $H$ \cite{Harary1969,Zykov1949}. Note that $G+H\cong K_2\circ\{G,H\}$. 

Moreover, complete $k$-partite graphs,  $K_{p_1,p_2,\dots,p_k}\cong K_n \circ \{N_{p_1},N_{p_2},\dots ,N_{p_k}\}\cong N_{p_1}+N_{p_2}+\cdots +N_{p_k}$, are typical examples of join graphs.


The lexicographic product graph $G\circ \mathcal{H}$ connected if and only if $G$ is connected and, in such a case, the  relation between distances in $G\circ \mathcal{H}$ and those in its factors is presented in the following remark.

\begin{remark}\label{remarkDistLexi}
If $G$ is a connected graph and $(u_i,b)$ and $(u_j,d)$ are vertices of $G\circ \mathcal{H}$, then
$$d((u_i,b),(u_j,d))=\left\{\begin{array}{ll}

                            d(u_i,u_j), & \mbox{if $i\ne j$,} \\
                            \\
                            d_2(b,d), & \mbox{if $i=j$}.
                            \end{array} \right.
$$
\end{remark}
We would point out that the remark above was stated in  \cite{Hammack2011,Imrich2000} for the case where $H_i\cong H$ for all $H_i\in \mathcal{H}$.

The lexicographic product has been studied from different points of view in the literature. For instance, the metric dimension and related parameters have been studied in \cite{Estrada-Moreno2014b,Feng2012a, JanOmo2012, Kuziak2014, Ramirez_Estrada_Rodriguez_2015, Saputro2013}. For more information on product graphs we suggest the books \cite{Hammack2011,Imrich2000}.

The following result allows to extend the  results on the $(k,t)$-metric dimension of lexicographic product graphs $G\circ\mathcal{H}$ to results on the $(k,2)$-metric dimension of $G\circ\mathcal{H}$, and vice versa.

\begin{theorem}\label{lexiSame_Metric_Adj_dim}
Let $G$ be a connected graph of order $n\ge 2$ and let $\mathcal{H}=\{H_1,\ldots,H_n\}$ be a family composed by nontrivial graphs, and $t\ge 2$ an integer. A set $A\subseteq V(G\circ\mathcal{H})$ is a $(k,t)$-metric generator for $G\circ\mathcal{H}$ if and only if $A$ is a $(k,2)$-metric generator for $G\circ\mathcal{H}$, and
as a consequence, $$\dim_{k}^t(G\circ\mathcal{H})=\dim_k^2(G\circ\mathcal{H}).$$
\end{theorem}

\begin{proof}
By definition, any $(k,2)$-metric generator for a graph is also a $(k,t)$-metric generator for $t\ge 2$.
Considering that any $(k,D(G))$-metric generator for a graph $G$ is also a $(k,t)$-metric generator for $t>D(G)$, we only need to prove that any $(k,D(G\circ\mathcal{H}))$-metric generator for $G\circ\mathcal{H}$ is also a $(k,2)$-metric generator.  For simplicity, we will use the terminology of $k$-metric generator and $k$-adjacency generator. Let $V(G)=\{u_1,\ldots,u_n\}$, let $S$ be a $k$-metric generator for $G\circ\mathcal{H}$, and let $S_i=S\cap (\{u_i\}\times V(H_i))$ for every $u_i\in V(G)$. We differentiate the following four cases for two different vertices  $(u_i,v),$ $(u_j,w)\in V(G\circ\mathcal{H})$.\\
\\Case 1. $i=j$. In this case $v\ne w$. By Remark \ref{remarkDistLexi}, no vertex from $S_l$, $l\ne i$, distinguishes $(u_i,v)$ and $(u_i,w)$. So it holds that $|\mathcal{D}((u_i,v),(u_i,w))\cap S_i|\ge k$. Since for any vertex $(u_i,x)\in S_i$ we have that $d((u_i,x),(u_i,v))=d_2((u_i,x),(u_i,v))$ and $d((u,x),(u_i,w))=d_2((u_i,x),(u_i,w))$, we conclude that $$k\le |\mathcal{D}_2((u_i,v),(u_i,w))\cap S_i|=|\mathcal{D}_2((u_i,v),(u_i,w))\cap S|.$$
Case 2. $i\ne j$ and $N[u_i]=N[u_j]$. By Remark \ref{remarkDistLexi}, no vertex from $S_l$, $l\notin\{i,j\}$, distinguishes $(u_i,v)$ and $(u_j,w)$. So $|\mathcal{D}((u_i,v),(u_j,w))\cap (S_i\cup S_j)|\ge k$. Since for any vertex $(u,x)\in S_i\cup S_j$ we have that $d((u,x),(u_i,v))=d_2((u,x),(u_i,v))$ and $d((u,x),(u_j,w))=d_2((u,x),(u_j,w))$, we conclude that $$k\le |\mathcal{D}_2 ((u_i,v), (u_j,w))\cap (S_i\cup S_j)|=|\mathcal{D}_2((u_i,v), (u_j,w))\cap S|.$$
Case 3. $i\ne j$ and $N(u_i)=N(u_j)$. This case is analogous to the previous one.\\
\\Case 4. $i\ne j$ and $u_i,u_j$ are not twins. Hence, there exists $u_l\in V(G)\setminus \{u_i,u_j\}$ such that $d_2(u_l,u_i)\ne d_2(u_l,u_j)$. Hence, for any vertex $(u_l,x)\in S_l$ we have that $$d_2((u_l,x),(u_i,v))=d_2((u_l,u_i)\ne d_2((u_l,u_j)=d_2((u_l,x),(u_j,w)).$$
According to Case $1$, we have that $|S_l|\ge k$. Therefore, we conclude that $$k\le |\mathcal{D}_2((u_i,v),(u_j,w))\cap S_l|\le |\mathcal{D}_2((u_i,v),(u_j,w))\cap S|.$$

In conclusion, $S$ is a $k$-adjacency generator for $G\circ\mathcal{H}$. The proof is complete.
\end{proof}

The reader is referred to \cite
{JanOmo2012,Saputro2013} for results on $\dim_1(G\circ\mathcal{H})$, and to  \cite{Estrada-Moreno2014b}  for results on $\dim_k(G\circ\mathcal{H})$, where $k\ge 2$.

\subsection{The case of corona product graphs}\label{coronas}

Let $G$ be a graph of order $n$ and let ${\cal H}=\{H_1,H_2,\ldots,H_n\}$ be a family of graphs. The \emph{corona product} graph $G\odot {\cal H}$\label{coroExtend}, introduced by  Frucht and Harary \cite{Frucht1970}, is defined as the graph obtained from $G$ and ${\cal H}$
by joining by an edge each vertex of $H_i$ with the ith vertex of $G$, for every $H_i\in \mathcal{H}$. Note that $G\odot {\cal H}$ is connected if and only if $G$ is connected. In particular, if the graphs in $\mathcal{H}$ are isomorphic to a given graph $H$, then we use the notation $G\odot H$ instead of  $G\odot {\cal H}$.

 The metric dimension and related parameters of corona product graphs have been studied in  \cite{Feng2012a, RV-F-2013, MR3218546, Yero2013d, Iswadi2011, Kuziak2013,  Rodriguez-Velazquez2013LDimCorona,RodriguezVelazquez,Yero2011}. In this subsection we will show that if $t\ge 3$ and $\mathcal{H}$ is composed by non-trivial graphs, then the $(k,t)$-metric dimension of $G\odot {\cal H}$ equals the sum of the $(k,t)$-metric dimensions of the graphs in $\mathcal{H}$.
In Section \ref{SectionComputational complexity} we will show that this strong relationship is an important tool to investigate the computational complexity of computing the $(k,t)$-metric dimension of graphs.

\begin{theorem}\label{kt-dim-corona}
Let $G$ be a connected graph of order $n\ge 2$, and $\mathcal{H}$ a family of $n$ non-trivial graphs. For any integers $t\ge 3$ and $k\ge 1$,
$$\dim_k^t(G\odot {\cal H})=\sum_{H\in \mathcal{H}}\dim_k^2(H).$$
\end{theorem}

\begin{proof}
We first introduce some notation. Let $V_0=\{u_1, u_2,\ldots, u_n\}$ be the vertex set of $G$, and let ${\cal H}=\{H_1,H_2,\ldots,H_n\}$.  For every $i\in \{1,\dots, n\}$, the vertex set of $H_i$ will be denoted by   $V_i$, so that the vertex set of $G\odot {\cal H}$ is $V=\bigcup_{i=0}^nV_i $.

If there exists a $(k,t)$-metric basis   $S$ for $G\odot {\cal H}$, then  $S\cap V_i$ is a $(k,2)$-metric generator for $H_i$, as no vertex outside of $V_i$ is able to distinguish two vertices in $V_i$ and $d_{t}(v,v')=d_{2}(v,v')$ for all $v,v'\in V_i$,
where the distance $d_t$ is taken on $G\odot {\cal H}$. Hence, $$\dim_k^t(G\odot {\cal H})=|S|\ge \sum_{i=1}^n|S\cap V_i|\ge \sum_{i=1}^n\dim_k^2(H_i).$$

We now proceed to show that $W=\bigcup_{i=1}^nW_i$  is a $(k,t)$-metric generator  for $G\odot {\cal H}$, where $W_i\subseteq V_i$ is a $(k,2)$-metric basis of $H_i$, for every $H_i\in \mathcal{H}$. To see this we differentiate the following cases for two different vertices $x,y\in V$.

\vspace{0,2cm}

\noindent Case 1: $x,y\in V_i$, $i\ne 0$. Since $W_i\subset W$ is a $(k,2)$-metric basis for $H_i$, and $d_{t}(v,v')=d_{2}(v,v')$ for all $v,v'\in V_i$, we can conclude that $|\mathcal{D}_t(x,y)\cap W|\ge k$.

\vspace{0,2cm}

\noindent Case 2: $x,y \in V_0$. Let $x=u_i$ and $y=u_j$. For any $z\in W_i$ we have $d_t(z,x)=1<d_t(z,y)$ and so $|\mathcal{D}_t(x,y)\cap W|\ge |W_i|\ge k$.

\vspace{0,2cm}

\noindent Case 3: $x \in V_i$ and $y\in V_j$, $i\ne j$. If $y=u_l\in V_0\setminus N_G(u_i)$ or $j\ne 0$, then for any $z\in W_i$ we have $d_t(z,x)\le 2<3\le d_t(z,y)$ and so $|\mathcal{D}_t(x,y)\cap W|\ge |W_i|\ge k$. If $y=u_l\in N_G(u_i)$, then for any $z\in W_l$ we have $d_t(z,y)\le 2<3\le d_t(z,x)$ and so $|\mathcal{D}_t(x,y)\cap W|\ge |W_l|\ge k$.

According to the three cases above we conclude that $W$  is a $(k,t)$-metric generator  for $G\odot {\cal H}$ and, as a consequence,
$$\dim_k^t(G\odot {\cal H})\le |W|= \sum_{i=1}^n|W_i|= \sum_{i=1}^n\dim_k^2(H_i),$$
as required. Therefore, if for every $H_i\in \mathcal{H}$ there exists a $(k,2)$-metric generator, then $\dim_k^t(G\odot {\cal H})= \sum_{i=1}^n\dim_k^2(H_i)$. On the other hand,
 if there exists $H_i\in \mathcal{H}$ such that no subset of $V_i$ is a $(k,2)$-metric generator for $H_i$, then no subset of $V$ is a $(k,t)$-metric generator for $G\odot {\cal H}$, so that  $\dim_k^2(H_i)=+\infty$ and $\dim_k^t(G\odot {\cal H})=+\infty$, which implies that $\dim_k^t(G\odot {\cal H})=\sum_{i=1}^n\dim_k^2(H_i)$.
\end{proof}

\section{Computational complexity}\label{SectionComputational complexity}

We next deal with the following decision problem, for which we prove its NP-completeness for the case in which $k$ is an odd integer.

$$\begin{tabular}{|l|}
  \hline
  \mbox{$(k,t)$-METRIC DIMENSION PROBLEM}\\
  \mbox{INSTANCE: A $(k',t)$-metric dimensional graph $G$ of order $n\ge 3$, integers $k,r$}\\
  \hspace*{2.4cm}\mbox{ with $2\le k\le t$ and  such that $2\le k\le k'$.}\\
  \mbox{QUESTION: Is $\dim_k^t(G)\le r$?}\\
  \hline
\end{tabular}$$\\

In order to study the problem above, we analyze its relationship with the two decision problems which are stated at next. We show that the first one of them is NP-complete, and for the second one, it is already known as an NP-complete problem from \cite{RV-F-2013}.

$$\begin{tabular}{|l|}
  \hline
  \mbox{$(k,2)$-METRIC DIMENSION PROBLEM}\\
  \mbox{INSTANCE: A $(k',2)$-metric dimensional graph $G$ of order $n\ge 3$ and an integer $k$}\\
  \hspace*{2.4cm} \mbox{such that $2\le k\le k'$.}\\
  \mbox{QUESTION: Is $\dim_k^2(G)\le r$?}\\
  \hline
\end{tabular}$$

$$\begin{tabular}{|l|}
  \hline
  \mbox{$(1,2)$-METRIC DIMENSION PROBLEM}\\
  \mbox{INSTANCE: A connected graph $G$ of order $n\ge 3$.}\\
  \mbox{QUESTION: Is $\dim_1^2(G)\le r$?}\\
  \hline
\end{tabular}$$\\

Since the problem above ($(1,2)$-METRIC DIMENSION PROBLEM) was proved to be NP-complete in \cite{RV-F-2013}, we shall proceed as follows. We first make a reduction from the $(1,2)$-METRIC DIMENSION PROBLEM to the $(k,2)$-METRIC DIMENSION PROBLEM, which shows the NP-completeness of this last mentioned problem. We further make a reduction from the $(k,2)$-METRIC DIMENSION PROBLEM to the $(k,t)$-METRIC DIMENSION PROBLEM, $t\ge 3$, showing the NP-completeness of our main problem.

We first consider a family of graphs $H_k$ constructed in the following way. Let $k$ be an odd integer and let $r=\frac{k-1}{2}$.
\begin{enumerate}
\item We begin with four vertices $a,b,c,d$ such that $a\sim b$ and $c\sim d$.
\item Add $r$ vertices $a_i$, $r$ vertices $c_i$, $k-1$ vertices $b_i$ and $k-1$ vertices $d_i$.
\item Add edges $aa_i$, $ba_i$, $cc_i$ and $dc_i$ with $i\in\{1,\ldots,r\}$ and edges $bb_j$ $dd_j$ with $j\in \{1,\ldots,k-1\}$.
\item Add edges $ab_{k-1}$ and $cd_{k-1}$.
\item For each vertex $w_l$ such that $w\in \{a,b,c,d\}$ and ($l\in \{1,\ldots,r\}$ or $l\in \{1,\ldots,k-1\}$ accordingly), add $r+1$ vertices $w_{l,q}$, $q\in \{1,\ldots,r+1\}$, and edges $w_lw_{l,q}$ for every $q\in \{1,\ldots,r+1\}$. For each $w_l$, we shall denote by $W_l$ (namely $A_l$, $B_l$, $C_l$ or $D_l$) the set of such vertices adjacent to $w_l$.
\item With all the vertices of the sets $A_l$'s and the sets $B_l$'s, construct a complete multipartite graph $K_{r+1,\ldots,r+1}$ having $k+r-1$ partite sets each of cardinality $r+1$ (each partite set given by a set $A_l$ or by a set $B_l$).
\item Similarly, proceed with the sets $C_l$'s and the sets $D_l$'s to obtain another complete multipartite graph $K_{r+1,\ldots,r+1}$.
\item For every $i\in \{1,\ldots,r\}$ and $j\in \{1,\ldots,r+1\}$, add the edges $a_{i,j}c_{q,j}$ with $q\in\{1,\ldots,r\}$.
\item For every $i\in \{1,\ldots,r\}$ and $j\in \{1,\ldots,r+1\}$, add the edges $b_{i,j}d_{q,j}$ with $q\in\{1,\ldots,r\}$.
\item For every $i\in \{r+1,\ldots,2r\}$ and $j\in \{1,\ldots,r+1\}$, add the edges $b_{i,j}d_{q,j}$ with $q\in\{r+1,\ldots,2r\}$.
\end{enumerate}
 Figure \ref{graph-H-5} shows an sketch of the graph $H_5$.

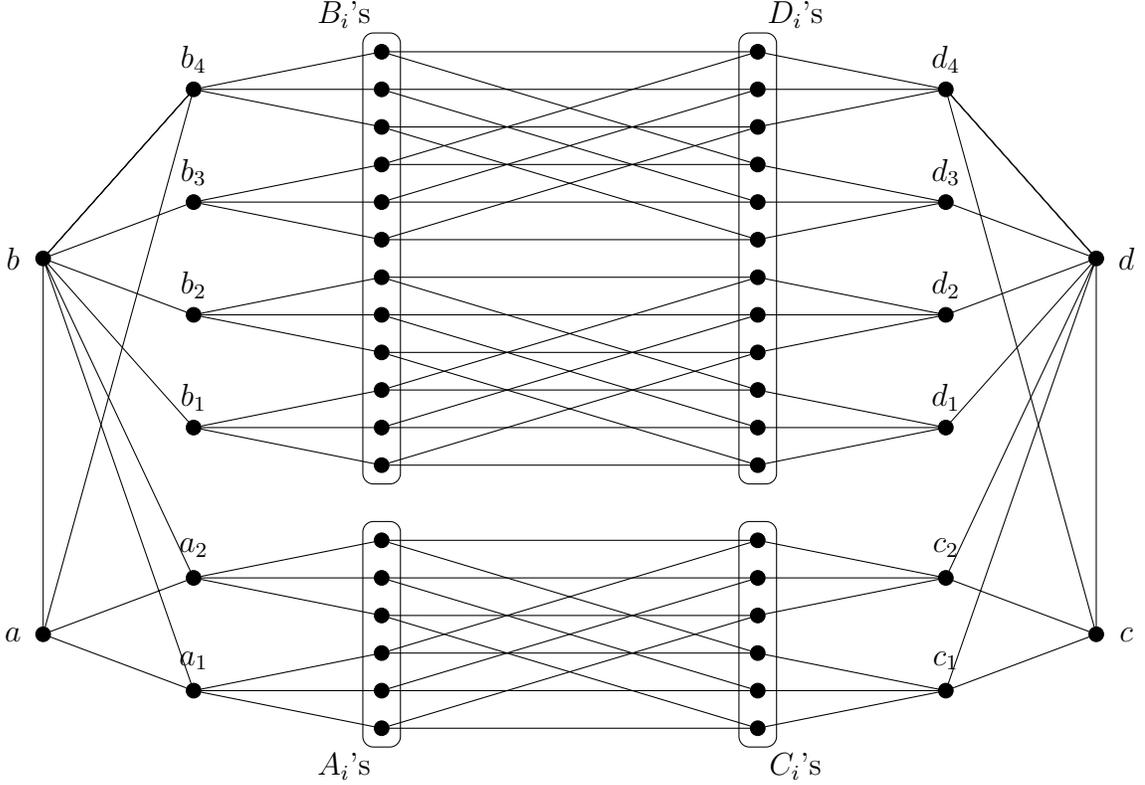
\begin{figure}[ht]
\centering
\begin{tikzpicture}[scale=.5, transform shape]
\node [draw, shape=circle,fill=black] (a) at  (-14,2.5) {};
\node [scale=2] at (-14.8,2.5) {$a$};
\node [draw, shape=circle,fill=black] (b) at  (-14,12.5) {};
\node [scale=2] at (-14.8,12.5) {$b$};
\node [draw, shape=circle,fill=black] (c) at  (14,2.5) {};
\node [scale=2] at (14.8,2.5) {$c$};
\node [draw, shape=circle,fill=black] (d) at  (14,12.5) {};
\node [scale=2] at (14.8,12.5) {$d$};

\node [draw, shape=circle,fill=black] (a1) at  (-10,1) {};
\node [scale=2] at (-10,1.8) {$a_1$};
\node [draw, shape=circle,fill=black] (a2) at  (-10,4) {};
\node [scale=2] at (-10,4.8) {$a_2$};
\node [draw, shape=circle,fill=black] (b1) at  (-10,8) {};
\node [scale=2] at (-10,8.8) {$b_1$};
\node [draw, shape=circle,fill=black] (b2) at  (-10,11) {};
\node [scale=2] at (-10,11.8) {$b_2$};
\node [draw, shape=circle,fill=black] (b3) at  (-10,14) {};
\node [scale=2] at (-10,14.8) {$b_3$};
\node [draw, shape=circle,fill=black] (b4) at  (-10,17) {};
\node [scale=2] at (-10,17.8) {$b_4$};

\node [draw, shape=circle,fill=black] (c1) at  (10,1) {};
\node [scale=2] at (10,1.8) {$c_1$};
\node [draw, shape=circle,fill=black] (c2) at  (10,4) {};
\node [scale=2] at (10,4.8) {$c_2$};
\node [draw, shape=circle,fill=black] (d1) at  (10,8) {};
\node [scale=2] at (10,8.8) {$d_1$};
\node [draw, shape=circle,fill=black] (d2) at  (10,11) {};
\node [scale=2] at (10,11.8) {$d_2$};
\node [draw, shape=circle,fill=black] (d3) at  (10,14) {};
\node [scale=2] at (10,14.8) {$d_3$};
\node [draw, shape=circle,fill=black] (d4) at  (10,17) {};
\node [scale=2] at (10,17.8) {$d_4$};

\node [draw, shape=circle,fill=black] (a11) at  (-5,0) {};
\node [draw, shape=circle,fill=black] (a12) at  (-5,1) {};
\node [draw, shape=circle,fill=black] (a13) at  (-5,2) {};
\node [draw, shape=circle,fill=black] (a21) at  (-5,3) {};
\node [draw, shape=circle,fill=black] (a22) at  (-5,4) {};
\node [draw, shape=circle,fill=black] (a23) at  (-5,5) {};
\node [draw, shape=circle,fill=black] (b11) at  (-5,7) {};
\node [draw, shape=circle,fill=black] (b12) at  (-5,8) {};
\node [draw, shape=circle,fill=black] (b13) at  (-5,9) {};
\node [draw, shape=circle,fill=black] (b21) at  (-5,10) {};
\node [draw, shape=circle,fill=black] (b22) at  (-5,11) {};
\node [draw, shape=circle,fill=black] (b23) at  (-5,12) {};
\node [draw, shape=circle,fill=black] (b31) at  (-5,13) {};
\node [draw, shape=circle,fill=black] (b32) at  (-5,14) {};
\node [draw, shape=circle,fill=black] (b33) at  (-5,15) {};
\node [draw, shape=circle,fill=black] (b41) at  (-5,16) {};
\node [draw, shape=circle,fill=black] (b42) at  (-5,17) {};
\node [draw, shape=circle,fill=black] (b43) at  (-5,18) {};

\node [draw, shape=circle,fill=black] (c11) at  (5,0) {};
\node [draw, shape=circle,fill=black] (c12) at  (5,1) {};
\node [draw, shape=circle,fill=black] (c13) at  (5,2) {};
\node [draw, shape=circle,fill=black] (c21) at  (5,3) {};
\node [draw, shape=circle,fill=black] (c22) at  (5,4) {};
\node [draw, shape=circle,fill=black] (c23) at  (5,5) {};
\node [draw, shape=circle,fill=black] (d11) at  (5,7) {};
\node [draw, shape=circle,fill=black] (d12) at  (5,8) {};
\node [draw, shape=circle,fill=black] (d13) at  (5,9) {};
\node [draw, shape=circle,fill=black] (d21) at  (5,10) {};
\node [draw, shape=circle,fill=black] (d22) at  (5,11) {};
\node [draw, shape=circle,fill=black] (d23) at  (5,12) {};
\node [draw, shape=circle,fill=black] (d31) at  (5,13) {};
\node [draw, shape=circle,fill=black] (d32) at  (5,14) {};
\node [draw, shape=circle,fill=black] (d33) at  (5,15) {};
\node [draw, shape=circle,fill=black] (d41) at  (5,16) {};
\node [draw, shape=circle,fill=black] (d42) at  (5,17) {};
\node [draw, shape=circle,fill=black] (d43) at  (5,18) {};

\draw[rounded corners] (-5.5,-0.5) rectangle (-4.5,5.5);
\node [scale=2] at (-6,-1) {$A_i$'s};
\draw[rounded corners] (-5.5,6.5) rectangle (-4.5,18.5);
\node [scale=2] at (-6,19) {$B_i$'s};
\draw[rounded corners] (4.5,-0.5) rectangle (5.5,5.5);
\node [scale=2] at (6,-1) {$C_i$'s};
\draw[rounded corners] (4.5,6.5) rectangle (5.5,18.5);
\node [scale=2] at (6,19) {$D_i$'s};

\draw(a)--(b)--(a1)--(a)--(a2)--(b)--(b1);
\draw(a)--(b4);
\draw(b2)--(b)--(b3);
\draw(b)--(b4)--(b);
\draw(b43)--(b4)--(b42);
\draw(b41)--(b4);
\draw(b33)--(b3)--(b32);
\draw(b31)--(b3);
\draw(b23)--(b2)--(b22);
\draw(b21)--(b2);
\draw(b13)--(b1)--(b12);
\draw(b11)--(b1);
\draw(a23)--(a2)--(a22);
\draw(a21)--(a2);
\draw(a13)--(a1)--(a12);
\draw(a11)--(a1);
\draw(c)--(d)--(c1)--(c)--(c2)--(d)--(d1);
\draw(d4)--(c);
\draw(d2)--(d)--(d3);
\draw(d)--(d4)--(d);
\draw(d43)--(d4)--(d42);
\draw(d41)--(d4);
\draw(d33)--(d3)--(d32);
\draw(d31)--(d3);
\draw(d23)--(d2)--(d22);
\draw(d21)--(d2);
\draw(d13)--(d1)--(d12);
\draw(d11)--(d1);
\draw(c23)--(c2)--(c22);
\draw(c21)--(c2);
\draw(c13)--(c1)--(c12);
\draw(c11)--(c1);
\draw(a11)--(c21)--(a21)--(c11)--(a11);
\draw(a12)--(c22)--(a22)--(c12)--(a12);
\draw(a13)--(c23)--(a23)--(c13)--(a13);
\draw(b11)--(d21)--(b21)--(d11)--(b11);
\draw(b12)--(d22)--(b22)--(d12)--(b12);
\draw(b13)--(d23)--(b23)--(d13)--(b13);

\draw(b31)--(d41)--(b41)--(d31)--(b31);
\draw(b32)--(d42)--(b42)--(d32)--(b32);
\draw(b33)--(d43)--(b43)--(d33)--(b33);

\end{tikzpicture}
\caption{An sketch of the graph $H_5$. The edges of the two complete bipartite subgraphs induced by the $A_i$'s and $B_i$'s, and by the $C_i$'s and $D_i$'s, have not been drawn.}\label{graph-H-5}
\end{figure}

We can easily check the following properties of $H_k$.

\begin{remark}\label{prop-graph-H_k}
For any graph $H_k$ the following follows.
\begin{enumerate}[{\rm (i)}]
\item $H_k$ has order $R=\frac{3k^2+6k-1}{2}$.
\item The degrees $\delta(v)$ of vertices $v\in V(H_k)$ are: $\delta(a)=\delta(c)=r+2$, $\delta(b)=\delta(d)=3r+1$, $\delta(a_i)=\delta(c_i)=r+3$ $($with $i\in\{1,\ldots,r\})$, $\delta(b_i)=\delta(d_i)=r+2$ $($with $i\in\{1,\ldots,k-2\})$, $\delta(b_{k-1})=\delta(d_{k-1})=r+3$,  $\delta(a_{i,j})=\delta(c_{i,j})=3r(r+1)$ $($with $i\in\{1,\ldots,r\}$ and $j\in\{1,\ldots,r+1\})$, and $\delta(b_{i,j})=\delta(d_{i,j})=3r(r+1)$ $($with $i\in\{1,\ldots,k-1\}$ and $j\in\{1,\ldots,r+1\})$.
\end{enumerate}
\end{remark}
We now study the $(k,2)$-metric dimension of the graph $H_k$ for any odd integer $k$.

\begin{remark}\label{dim_k-2-H_k}
For any graph $H_k$ with $k$ being an odd integer, $\dim_k^2(H_k)=R-6$.
\end{remark}

\begin{proof}
Let $S$ be a $(k,2)$-metric basis of $H_k$. Notice that $|\mathcal{D}_2(a,b)|=|\{a,b,b_1,\ldots,b_{k-2}\}|=k$ and $|\mathcal{D}_2(c,d)|=|\{c,d,d_1,\ldots,d_{k-2}\}|=k$, which implies that $\{a,b,b_1,\ldots,b_{k-2}\}\subset S$ and $\{c,d,d_1,\ldots,d_{k-2}\}\subset S$. On the other hand, notice that for any two vertices $b_{i,j},b_{i,q}$ with $i\in \{1,\ldots,k-1\}$ and $j,q\in\{1,\ldots,r+1\}$, it follows $|\mathcal{D}_2(b_{i,j},b_{i,q})|=k+1$. In this sense, if $|\left(\bigcup_{i=1}^r B_i\cup D_i\right)-S|\ge 2$, then there are at least two vertices $b_{i,j},b_{i,q}$ or at least two vertices $d_{i,j},d_{i,q}$ for some $i\in \{1,\ldots,k-1\}$ and $j,q\in\{1,\ldots,r+1\}$ for which $|\mathcal{D}_2(b_{i,j},b_{i,q})|\le k-1$ or $|\mathcal{D}_2(d_{i,j},d_{i,q})|\le k-1$, respectively, and this is not possible. Thus $|\left(\bigcup_{i=1}^r B_i\cup D_i\right)-S|\le 1$. Similarly, we observe that $|\left(\bigcup_{i=r+1}^{2r} B_i\cup D_i\right)-S|\le 1$ and $|\left(\bigcup_{i=1}^r A_i\cup C_i\right)-S|\le 1$. Consequently, at most three vertices of the sets $A_l$'s, $B_l$'s, $C_l$'s and $D_l$'s do not belong to $S$.

Now, we note that $|\mathcal{D}_2(a,a_i)|=k+2$, $i\in \{1,\ldots,r\}$, which means that at most two vertices of the set $\{a,b_{k-1},a_1,\ldots,a_r\}\subseteq \mathcal{D}_2(a,a_i)$ do not belong to $S$. Similarly, at most two vertices of the set $\{c,d_{k-1},c_1,\ldots,c_r\}$ do not belong to $S$. If exactly two vertices of the set $\{a,b_{k-1},a_1,\ldots,a_r\}$ do not belong to $S$, then $|\left(\bigcup_{i=1}^r A_i\right)-S|=0$, otherwise there is a vertex $a_j$ for which $|\mathcal{D}_2(a,a_j)|<k$. A similar reasoning can be deduced for the set of vertices $\{c,d,d_{k-1},c_1,\ldots,c_r\}$. Consequently, we have either one of the following situations.
\begin{itemize}
\item Exactly two vertices of the set $\{a,b_{k-1},a_1,\ldots,a_r\}$ do not belong to $S$, $|\left(\bigcup_{i=1}^r A_i\right)-S|=0$, $|\left(\bigcup_{i=1}^r C_i\right)-S|=1$ and at most one vertex of the set $\{c,d,d_{k-1},c_1,\ldots,c_r\}$ do not belong to $S$, or
\item at most one vertex of the set $\{a,b,b_{k-1},a_1,\ldots,a_r\}$ do not belong to $S$, $|\left(\bigcup_{i=1}^r A_i\right)-S|=1$, $|\left(\bigcup_{i=1}^r C_i\right)-S|=0$ and exactly two vertices of the set $\{c,d_{k-1},c_1,\ldots,c_r\}$ do not belong to $S$, or
\item exactly two vertices of the set $\{a,b_{k-1},a_1,\ldots,a_r\}$ do not belong to $S$, $|\left(\bigcup_{i=1}^r A_i\cup C_i\right)-S|=0$ and exactly two vertices of the set $\{c,d_{k-1},c_1,\ldots,c_r\}$ do not belong to $S$.
\end{itemize}
Notice that it cannot happen: at most one vertex of the set $\{a,b,b_{k-1},a_1,\ldots,a_r\}$ do not belong to $S$, $|\left(\bigcup_{i=1}^r A_i\right)-S|=1$, $|\left(\bigcup_{i=1}^r C_i\right)-S|=1$ and at most one vertex of the set $\{c,d_{k-1},c_1,\ldots,c_r\}$ do not belong to $S$, since in such case $|\left(\bigcup_{i=1}^r A_i\cup C_i\right)-S|=2$, which is not possible, as stated before.

In any of the situations previously described, we can deduce that at most four vertices in the set $\{a,b_{k-1},a_1,\ldots,a_r\}\cup \{c,d_{k-1},c_1,\ldots,c_r\}\cup \left(\bigcup_{i=1}^r A_i\cup C_i\right)$ do not belong to $S$. Finally, since at most two vertices of the sets $B_l$'s and $D_l$'s do not belong to $S$, we obtain that at most six vertices of $H_k$ do not belong to $S$, or equivalently, $\dim_k^2(H_k)=|S|\ge R-6$.

Now, let $S'=V(H_k)-\{a_1,c_1,b_{k-1},d_{k-1},d_{1,1},d_{r+1,1}\}$. We will show that $S'$ is a $(k,2)$-metric generator for $H_k$. To this end, we consider the following table containing lower bounds for the value $|\mathcal{D}_2(x,y)\cap S'|$ for some pairs of vertices $x,y\in V(H_k)$ (in some cases the bounds are not the best ones, but enough to prove what we need).

\begin{center}
\begin{tabular}{|c|c|c|c|c|c|c|}
  \hline
   & $a$ & $b$ & $a_f$ & $b_g$ & $a_{i,j}$ & $b_{l,q}$ \\ \hline
  $a$ & - & $k$ & $k$ & $k$ & $k(r+1)$ & $k(r+1)$ \\ \hline
  $b$ & $k$ & - & $k-1+r$ & $k+1$ & $k(r+1)$ & $k(r+1)$ \\ \hline
  $a_f$ & $k$ & $k-1+r$ & - & $k$ & $k(r+1)$ & $k(r+1)$ \\ \hline
  $b_g$ & $k$ & $k+1$ & $k$ & - & $k(r+1)$ & $k(r+1)$ \\ \hline
  $a_{i,j}$ & $k(r+1)$ & $k(r+1)$ & $k(r+1)$ & $k(r+1)$ & - & $k+1$ \\ \hline
  $b_{l,q}$ & $k(r+1)$ & $k(r+1)$ & $k(r+1)$ & $k(r+1)$ & $k+1$ & - \\
  \hline
\end{tabular}
\end{center}
On the other hand, $|\mathcal{D}_2(a_i,a_j)\cap S'|\ge k+2$, $|\mathcal{D}_2(b_i,b_j)\cap S'|\ge k+1$, $|\mathcal{D}_2(a_{i,j},a_{i,q})\cap S'|\ge k$, $|\mathcal{D}_2(a_{i,j},a_{l,q})\cap S'|\ge k+2$ ($l\ne i$), $|\mathcal{D}_2(b_{i,j},b_{i,q})\cap S'|\ge k$ and $|\mathcal{D}_2(b_{i,j},b_{l,q})\cap S'|\ge k+3$ ($l\ne i$).

A similar table and similar results as above can be done for vertices of type $c,d$. So, it remains only those pairs of vertices such that one of them is of type $a,b$ and the other one of type $c,d$. For instance, $|\mathcal{D}_2(a,c)\cap S'|\ge k+1$, $|\mathcal{D}_2(a,d)\cap S'|\ge k+r$, $|\mathcal{D}_2(b,c)\cap S'|\ge k+r$, $|\mathcal{D}_2(b,d)\cap S'|\ge 2k+2r-2$. The remaining cases are left to the reader.

As a consequence of the situations described above, we have that $S'$ is a $(k,2)$-metric generator for $H_k$. Therefore, $\dim_k^2(H_k)\le R-6$ and the equality follows.
\end{proof}

In order to continue our exposition, we assume some notations. According to the definition of corona product graphs $G\odot H$ (whether all the graphs in the family $\mathcal{H}$ are isomorphic to a graph $H$) given in Subsection \ref{coronas}, from now on we will denote by $U=\{u_1,u_2,\ldots, u_{n}\}$ the set of vertices of $G$ and by $V_i$ the vertices of $H_i$, $i\in\{1,\ldots,n\}$. Moreover, the vertex set of $G\odot H$ is given by $\{u_1,u_2,\ldots, u_{n}\}\cup \{(u_i,v_j)\;:\;u_i\in U,\;\;v_j\in V_j\}$. Now, given any connected graph $G$ and an odd integer $k$, we shall construct a graph $G'$ in the following way.

\begin{enumerate}
\item Consider the corona product graph $G\odot N_r$ where $N_r$ is the empty graph on $r$ vertices (recall $r=\frac{k-1}{2}$).
\item For any vertex $(u_i,v_j)\in V(G\odot N_r)$ such that $u_i\in U$ and $v_j\in V_j$, add a copy of the graph $H_k$ and identify the vertex $(u_i,v_j)$ of $G\odot N_r$ with the vertex $b_1$ in the copy of $H_k$.
\end{enumerate}

We are now able to prove that the $(k,2)$-METRIC DIMENSION PROBLEM is NP-complete, for  $k$   odd.

\begin{theorem}
For any odd integer $k$, the $(k,2)$-METRIC DIMENSION PROBLEM is NP-complete.
\end{theorem}

\begin{proof}
It is not difficult to observe that the problem is in NP, since verifying that a given set is a $(k,2)$-metric generator can be done in polynomial time. Let $G$ be any non-trivial graph. We consider the graph $G'$ as described above and will prove that
\begin{equation}\label{EqComplexity}
\dim_k^2(G')=\dim_1^2(G)+\frac{n(k-1)}{2}\dim_k^2(H_k) .
\end{equation}
 Let $S_k$ be a $(k,2)$-metric basis for $H_k$ as described in the second part of the proof of Remark \ref{dim_k-2-H_k}.  Let $S_G$ be any $(1,2)$-metric basis for $G$ and let $S_H$ be the union of the sets $S_k$ corresponding to the copies  of $H_k$. In order to show that  $S=S_G\cup S_H$ is a $(k,2)$-metric generator for $G'$, we analyze the following cases for any pair of  different   vertices $x,y$ of $G'$.\\

\noindent Case 1. $x,y\in V(G)$. Since every vertex in $V(G)$ is adjacent to $r=\frac{k-1}{2}$ vertices of $S$ and also $|\mathcal{D}_{G,2}(x,y)|\ge 1$, it clearly follows that $|\mathcal{D}_{G',2}(x,y)\cap S|\ge 2r+1=k$.\\

\noindent Case 2. $x\in V(G)$ and $y\notin V(G)$. According to the degrees of vertices of $H_k$ (see Remark \ref{prop-graph-H_k} (ii)) and the structure of $S_k$, we notice that if $y\not\in \{a,c\}$, then $|N[y]\cap S|\ge r+2$ and also, by the construction of $G'$, $|(N[x]\cap S)-N(y)|\ge r-1$.  Now, if $y\in \{a,c\}$,   then $|N[y]\cap S|\ge r+1$ and  $|(N[x]\cap S)-N(y)|\ge r$. Since in both cases  $(N[y]\cap S)\cap ((N[x]\cap S)-N(y))=\emptyset$, it follows $|\mathcal{D}_{G',2}(x,y)\cap S|\ge 2r+1 = k$.\\

\noindent Case 3. $x,y\notin V(G)$. If $x,y$ belong to two different copies of $H_k$,  then  $|N[x]\cap S|\ge r+1$, $|N[y]\cap S|\ge r+1$ and $(N[y]\cap S)\cap (N[x]\cap S)=\emptyset$. Thus, $|\mathcal{D}_{G',2}(x,y)\cap S|\ge 2r+2 > k$. Now, if $x,y$ belong to the same copy of $H_k$, then $|\mathcal{D}_{G',2}(x,y)\cap S|\ge |\mathcal{D}_{G',2}(x,y)\cap S_k|\ge k$. \\

According to the cases above, it clearly follows that $S$ is a $(k,2)$-metric generator for $G'$ and so, $\dim_k^2(G')\le \dim_1^2(G)+\frac{n(k-1)}{2}\dim_k^2(H_k)$.\\

Now, consider a $(k,2)$-metric basis $S'$ of $G'$. Let $S''=S'\cap V(G)$ and let $u,v\in V(G)$.  For the vertices  $a,b,b_1,b_2,\ldots,b_{k-2}$, corresponding to  a copy of $H_k$, we have $\mathcal{D}_{G',2}(a,b)=\{a,b,b_1,b_2,\ldots,b_{k-2}\}$ and, as a consequence, the vertex $b_1$ corresponding to each copy of $H_k$ must belong to $S'$. Hence, $|\mathcal{D}_{G',2}(u,v)\cap (S'-S'')|=2r=k-1$, which implies that $S''$ must be a $(1,2)$-metric generator for $G$.
Furthermore, as we have shown in the proof of Remark \ref{dim_k-2-H_k},
to ensure that a set $D\subseteq V(H_k)$ satisfies
$ \left| D\cap \mathcal{D}_{H_k,2}(x,y)\right |\ge k$, for any pair of vertices $x,y\in V(H_k)-\{b_1\}$, the cardinality of  $D$ must be greater than or equal to
$R-6=\dim_k^2(H_k)$, which implies that $|S'\cap V(H_k)|\ge \dim_k^2(H_k)$ for all copies of $H_k$.
As a consequence,
$$\dim_k^2(G')=|S'\cap V(G')|+\sum_{i=1}^{\frac{n(k-1)}{2}}|S'\cap V(H_k)|
\ge \dim_1^2(G)+\frac{n(k-1)}{2}\dim_k^2(H_k)$$
 and  \eqref{EqComplexity} follows.
 The reduction from the $(1,2)$-METRIC DIMENSION PROBLEM to the $(k,2)$-METRIC DIMENSION PROBLEM is deduced by Remark \ref{dim_k-2-H_k}  and \eqref{EqComplexity}.
\end{proof}

Our next step is the proof of the NP-completeness of our main problem: the $(k,t)$-METRIC DIMENSION PROBLEM. To this end, we shall use a result already presented in Subsection \ref{coronas}.

\begin{theorem}\label{complex-kt-dim-k-odd}
For any odd integer $k$ and any integer $t\ge 2$, the $(k,t)$-METRIC DIMENSION PROBLEM is NP-complete.
\end{theorem}

\begin{proof}
 Since verifying that a given set is a $(k,t)$-metric generator can be done in polynomial time,  the problem is in NP. Consider now any non-trivial graph $H$ and let $G$ be any connected graph of order $n\ge 2$. By Theorem \ref{kt-dim-corona}, $\dim_k^t(G\odot H)=n\dim_k^2(H)$. Thus, the reduction from the $(k,2)$-METRIC DIMENSION PROBLEM to the $(k,t)$-METRIC DIMENSION PROBLEM is deduced, and the proof is completed.
\end{proof}

\section{Concluding remarks and future
works}\label{SectionFutureWorks}

In this section we discuss some problems which are derived from  or  related to  our previous results. All these problems deserve a deeper study than we have yet given them.

\begin{itemize}

\item {\bf $(k,t)$-metric dimensional graphs.}

In Section~\ref{Sect(k,t)dimensional} we have discussed a natural problem in the study of the $k$-metric dimension of a metric
space $(X,d_t)$ which consists of finding the largest integer $k$ such that
there exists a $k$-metric generator for $X$. We have shown that, when restricted to the case of an specific graph, the problem is very easy to solve. Even so, Theorem \ref{TrivialUpperBound} shows that from a theoretical point of view it would be desirable to obtain some general results on this subject.

\item {\bf Computing the $(k,t)$-metric dimension.}

In Section~\ref{bound-k_DimG} we give some basic bounds on $\dim_k^t(G)$ and  discuss the extreme cases. We also show that there are some families of graphs having the same $(k,t)$-metric dimension and give the value of $\dim_k^t(G)$ for some particular cases.  It would be desirable to obtain specific results on $\dim_k^t(G)$ for graphs satisfying certain restriction, \emph{i.e.}, the case of product graphs. We would emphasize that the problem of computing the $(k,2)$-metric dimension ($k$-adjacency dimension) of corona product graphs remains open.

\item {\bf The simultaneous metric dimension of metric spaces}.

Given a family $\mathcal{X}=\{(X,d^{(1)}),(X,d^{( 2)}), \ldots ,(X,d^{(r)})\}$ of metric spaces, we define a \textit{simultaneous $k$-metric generator} for ${\mathcal{X}}$ to be a set $S\subseteq X$ such that $S$ is simultaneously a $k$-metric generator for each metric space $(X,d^{(i)})$. We say that a smallest simultaneous $k$-metric generator for ${\mathcal{X}}$ is  a \emph{simultaneous $k$-metric basis} of ${\mathcal{X}}$, and
its cardinality the \emph{simultaneous $k$-metric dimension} of ${\mathcal{X}}$, denoted by $\Sd_k({\mathcal{X}})$. The simultaneous $1$-metric dimension was introduced in \cite{Ramirez-Cruz-Rodriguez-Velazquez_2014}, where the families of metrics spaces are composed by graphs defined on the same vertex set, which are equipped  with the geodesic distance.

We now illustrate this with three examples.

\noindent {\bf Example 1.}
According to Theorem \ref{lexiSame_Metric_Adj_dim} we can claim that for any connected graph $G$ and
any family $\mathcal{H}$ of non-trivial graphs, the  family $\mathcal{X}$ of metric spaces obtained from a graph $G\circ\mathcal{H}$ equipped with the metrics $d_2, d_3, \ldots$,    has simultaneous $k$-metric dimension $\Sd_k({\mathcal{X}})=\dim_k^2(G\circ\mathcal{H})$.

\noindent{\bf Example 2.}
By Theorem \ref{Common_k_t_metric_generator}
we have that the family of graphs ${\mathcal{G}}_B(G)$ defined in Section \ref{SectionFamilies(k,t)Metric} equipped with the metric $d_t$ has simultaneous $k$-metric dimension $\Sd_k({\mathcal{G}})= \dim_{k}^t(G).$

\noindent{\bf Example 3.}
For many reasons in mathematics it is often convenient to work with bounded  distances. For instance, there is a simple mechanism to convert a given distance function $d(x,y)$ into (in a sense, equivalent) a bounded distance function   $d^{(i)}(x, y) = \frac{d(x,y)}{1+i d(x,y)}$, where $i$ is a positive integer. Consider a metric space $(X,d)$ and the associated family of metric spaces ${\mathcal{X}}=\{(X,d), (X,d^{(1)}),(X,d^{(2)}),\ldots\}$.
Let $x,y,z\in X$. Then   $d(x,y)\ne d(x,z)$ if and only if   $ d^{(i)}(x,y)\ne d^{(i)}(x,z)$, for all integers $i\ge 1$.  Hence,  any $k$-metric generator of $(X,d)$ is a $k$-metric generator of $(X,d^{(i)})$ and vice versa. Therefore,
$\Sd_k({\mathcal{X}})$ equals the $k$-metric dimension of $(X,d)$.

\item {\bf The lexicographic product of metric spaces}.

The lexicographic product of two metric spaces can be defined in a similar way to the lexicographic product of two graphs.
Let $(X,d)$ be a metric space. If there exists $t>0$ such that $$\min_{x,x'\in X, x\ne x'}d(x,x')=\frac{t}{2},$$ then the lexicographic product of  $(X,d)$  and a metric space $(Y,d')$   is the metric space $(X \times Y,\rho )$, where

$$\rho((x,y),(x',y'))=\left\{\begin{array}{ll}

                            d(x,x'), & \text{ if } x\ne x', \\
                            \\
                          \displaystyle \min\left\{ d'(y,y'),2\min_{z\in X\setminus \{x\}}d(x,z)\right\}, & \text{ if }  x=x'.
                            \end{array} \right.
$$

As with graphs, $X\circ Y$ always represents the metric space $(X \times Y,\rho )$, where in
this case $t$ will be understood from the context. As in the case of graphs, the definition above can be generalised to the product of a metric space times a family of metric spaces.

For any $x,x'\in X$ such that $d(x,x')=\frac{t}{2}$ and any $k$-metric generator $W$ of $X \circ Y $, the restriction of $W$ to $\{x\}\times Y$ induces a $(k,t)$-metric generator for $Y$, as two vertices in $ \{x\}\times Y$ are not distinguished by vertices outside of $\{x\}\times Y$, which implies that the projection of $W$ on $Y$, $W_Y=\{y:\, (x,y)\in W\}$, is a $(k,t)$-metric generator for $Y$. Hence, by Theorem~\ref{Dim^tUnboundedSpaces} we can conclude that if $(Y,d')$ is unbounded, then $\dim_k(X \circ Y)=+\infty$. This means that the study of the $k$-metric dimension of $X \circ Y $ should be restricted to cases where the second factor is bounded.

As an example we consider a simple and connected graphs $G=(V,E)$ of order $n$
 and a (non-necessarily bounded) metric space $(Y,d')$, where $|Y|\ge 2$. Then we construct the lexicographic product $G\circ Y$ from  the graph $G$ and the metric space $(Y,d^{(1)})$ equipped with the metric $d^{(1)}=\frac{d'}{1+d'}$. In this case, it is not difficult to check that
the $k$-metric dimension of $G\circ Y$ equals $n$ times the $(k,t)$-metric dimension of $(Y,d^{(1)})$. We leave
the details to the reader.

\item {\bf Computational complexity.}

Theorem \ref{complex-kt-dim-k-odd} allows to claim that computing the $(k,t)$-metric dimension of graphs is NP-hard for the case in which $k$ is an odd integer. It is probably not surprising that the case $k$ even has similar complexity. However, this case remains open and it would be interesting to complete this study. Moreover, in concordance with the NP-hardness of the problem, it might deserve to develop some approximation algorithms for this general approach as those ones already known for the standard metric dimension (see for instance \cite{Hartung2013} and \cite{Hauptmann2012}).

\item {\bf Practical applications.}

The metric dimension of graphs in its standard version ($(k,t)$-metric dimension with $k=1$ and $t$ as the diameter of $G$) has found in the last decades several applications to practical problems. For instance, the author of \cite{Johnson1993,Johnson1998} rediscovered the concept of metric dimension while  investigating some topics in chemistry. Applications to problems of pattern recognition and image processing appeared in \cite{Melter1984} and to navigation of robots in networks in \cite{Chartrand2000,Hulme1984,Khuller1996}. Also, some connections between the metric dimension and the Mastermind game or coin weighing have been presented in \cite{Caceres2007}. Furthermore,  in the recent work \cite{Bailey2016}, the $k$-metric dimension ($(k,t)$-metric dimension with $k>1$ and $t$ as the diameter of $G$) has found an interesting application while designing error-correcting codes. In this sense, it is natural to look also for some possible applications for this new general approach introduced in this work.

\end{itemize}

\bibliographystyle{elsart-num-sort}

\end{document}